\documentclass[11pt]{article}
\catcode`\@=11 \catcode`\@=12

\usepackage{amsmath}
\usepackage{amssymb}
\usepackage{amsthm}
\usepackage{multicol}

\usepackage[pdftex,colorlinks,urlcolor=blue,pdfstartview=FitH]{hyperref}

\newcommand{\set}[2]{\{\,\,#1\,\,|\,\,#2\,\}}

\newcommand{\cp}{\mathbb CP}

\newcommand{\Z}{\mathbb Z}
\newcommand{\R}{\mathbb R}
\newcommand{\C}{\mathbb C}
\newcommand{\F}{\mathbb F}

\newcommand{\Mr}{M_\mathbb R}
\newcommand{\Ker}{\mathrm{Ker}}
\newcommand{\Ima}{\mathrm{Im}}

\newcommand{\unov}{\Lambda^{\Z_2}}
\newcommand{\unovo}{\Lambda^{\mathbb{Z}_2}_0}
\newcommand{\unovf}{\Lambda^{\mathbb{F}_2}}
\newcommand{\unovof}{\Lambda^{\mathbb{F}_2}_0}

\newtheorem{theorem}{Theorem}[section]
\newtheorem{lemma}[theorem]{Lemma}
\newtheorem{definition}[theorem]{Definition}
\newtheorem{corollary}[theorem]{Corollary}
\newtheorem{proposition}[theorem]{Proposition}
\newtheorem{example}[theorem]{Example}
\newtheorem{rmk}[theorem]{Remark}

\def\leq{\leqslant}
\def\geq{\geqslant}

\title{Floer Cohomology of Torus Fibers and Real Lagrangians in Fano Toric Manifolds}
\author{Garrett Alston 
\hspace{1cm} Lino Amorim}
\date{}
\begin{document}

\maketitle 
\begin{multicols}{2}

\begin{small}
\noindent
Garrett Alston,
 Department of Mathematics\\
 Kansas State University,\\
Manhattan, KS,\\
USA\\
\texttt{galston@math.ksu.edu}\\

\noindent
Lino Amorim,
 Department of Mathematics,\\
University of Wisconsin-Madison\\
Madison, WI,\\
USA\\
\texttt{amorim@math.wisc.edu}\\

\end{small}

\end{multicols}

\vspace{1cm}

\textbf{Abstract } 
This article deals with the Floer cohomology (with $\Z_2$ coefficients) between torus
fibers and the real Lagrangian in Fano toric manifolds. We first investigate the
conditions under which the Floer cohomology is defined, and then  develop a combinatorial
description of the Floer complex based on the polytope of the toric manifold. This description is used to show
that if the Floer cohomology is defined, and the Floer cohomology of the torus fiber is
non-zero, then the Floer cohomology of the pair is non-zero. Finally, we
develop some applications to non-displaceability and the minimum number of intersection
points under Hamiltonian isotopy. 

\vspace{1.5cm}

\section{Introduction}
The Lagrangian Floer theory of torus fibers in toric manifolds has been extensively studied in the literature, first in the Fano case by Cho and Oh in \cite{chooh} and later in complete generality by Fukaya, Oh, Ohta and Ono in \cite{toricfooo}.
One of the fundamental results of their work is that for each torus fiber $L_c$ and flat unitary line bundle $\mathcal L_\rho$ on $L_c$, the Floer cohomology $HF(L_c,\mathcal L_\rho)$ is well-defined.
In the terminology of \cite{fooo}, the objects $(L_c,\mathcal L_\rho)$ are weakly unobstructed.
Furthermore, a function $W$ called the potential function is defined on the set of pairs $(L_c,\mathcal L_\rho)$, and the Floer cohomology of $(L_c,\mathcal L_\rho)$ is non-zero if and only if the gradient of $W$ evaluated at $(L_c,\mathcal L_\rho)$ is zero.

Toric manifolds possess another natural Lagrangian submanifold, namely the real part--the fixed point set of complex conjugation. When the toric manifold is Fano the real Lagrangian is an example of the fixed point set of an anti-symplectic involution in a spherically positive symplectic manifold. These have been studied in \cite{fooo8}, where it is shown that such a Lagrangian is unobstructed (over $\Z_2$) and the Floer cohomology is isomorphic to its singular cohomology.

In this article, we investigate the conditions under which the Floer cohomology between a torus fiber $L_c$ and the real Lagrangian $R$ is well-defined. 
In case it is, we describe the Floer complex in a purely combinatorial way based on the moment polytope of the toric manifold.
The result is Theorem \ref{main} and is the main result of this paper.
An interesting feature of our construction is that we twist the usual Floer cochain complex by a locally constant sheaf $\mathcal L_\rho$ with values in $\mathbb F_2$, the algebraic closure of $\mathbb Z_2$.

The description of the Floer complex we give is obtained by counting Maslov index 1 strips that connect intersection points of $R$ and $L_c$.
This is the standard way of defining Floer cohomology.
However it applies only to monotone Lagrangian submanifolds, and $R$ and $L_c$ are not monotone.
The standard proofs therefore do not work to show that the Floer cohomology is well-defined and invariant under choice of almost complex structure and Hamiltonian isotopy.
Thus our approach to the Floer theory of $(R,L_c)$ is to use the more general definition given in \cite{fooo} and then show that still the Floer complex really does just count Maslov index 1 strips (for the standard toric complex structure).

The combinatorial description of Floer cohomology is then used to get some results on the minimal number of intersection points between $R$ and $L_c$ under Hamiltonian isotopy.
This builds on the work of \cite{alston}, as well as recovers some non-displaceability results obtained elsewhere (\cite{bc}, \cite{bc1}, \cite{ep} and \cite{t}).
For example, consider the Lagrangian submanifolds $\R P^k$ and $T^k$ (the Clifford torus) in $\cp^k$.
We prove
\begin{theorem}
Let $\phi$ be a Hamiltonian diffeomorphism such that $T^k$ and $\phi(\R P^k)$ intersect transversely. Then
$$\sharp \big( \phi ( \mathbb{R}P^{k} ) \cap T^{k} \big) \geq 2^{ \big\lceil \frac{k}{2} \big\rceil }. $$
\end{theorem}

We do not have a general formula for the rank of the Floer cohomology.
However, motivated by \cite{chooh} and \cite{toricfooo}, we associate a potential function $W_c$ to each torus fiber $L_c$ and then are able to show 
\begin{theorem}
The Floer cohomology of the pair $(R,(L_c,\mathcal L_\rho))$ is well-defined if and only if $W_c(\rho)=0$.
Furthermore, $HF(R,(L_c,\mathcal L_\rho)) \neq 0$ if and only if $\nabla W_c(\rho)=0$.
\end{theorem}
This fact explains why we use values in $\mathbb F_2$ rather than $\mathbb Z_2$: $\mathbb Z_2$ is usually not big enough to find solutions to the equations $W_c=0, \nabla W_c=0$.
As mentioned before, it is shown in \cite{chooh} that the condition $\nabla W_c(\rho)=0$ is equivalent to the non-vanishing of $HF(L_c,\mathcal L_\rho)$.
Therefore Theorem 1.2 implies that, when defined, the cohomology $HF(R,(L_c,\mathcal L_\rho)) \neq 0$ if and only if $HF(L_c,\mathcal L_\rho) \neq0$.

When $\nabla W_c(\rho)=0$ but $W_c(\rho) \neq 0$, the Floer cohomology $HF(R,(L_c,\mathcal L_\rho))$ is not defined. Nevertheless, we can still show that $R$ and $L_c$ are non-displaceable. Following an idea of Miguel Abreu and Leonardo Macarini we consider the product of the toric manifold with itself. Then $HF(R \times R, (L_c\times L_c,\mathcal L_\rho\oplus\mathcal L_\rho))$ is well defined and we obtain
\begin{theorem}\label{product}
Suppose there exists a locally constant sheaf $\mathcal L_\rho$ such that $\nabla W_c(\rho)=0$ and let $\phi$ be a Hamiltonian diffeomorphism such that $\phi(R)$ and $L_c$ intersect transversely. Then
$$ \sharp \big( \phi(R) \cap L_c \big) \geq \sqrt{rank (HF(R \times R, (L_c\times L_c,\mathcal L_\rho \oplus \mathcal L_\rho)))} > 0. $$
\end{theorem}

The fact that $\mathbb F_2$ is algebraically closed allows us to obtain one more result.
\begin{corollary}
Suppose $X_P$ is monotone and $L_0$ is the unique monotone torus fiber. Then there exists $\mathcal L_\rho$ such that $\nabla W_0(\rho)=0$. Therefore the Lagrangians $R$ and $L_0$ are non-displaceable, i.e. for any Hamiltonian diffeomorphism $\phi$,
$$\phi(R) \cap L_c \neq \emptyset .$$
\end {corollary}

In \cite{toricfooo} the authors consider a similar potential function $W_c$ whose domain includes the space of all weak bounding cochains in $L_c$. 
The additional flexibility gained by considering all weak bounding cochains allows them to show that if $X_P$ is Fano then there for at least one torus fiber $L_c$ the function $W_c$ has a critical point.

We briefly describe the organization of this article.
Section 2 contains a brief summary of the construction of a toric manifold, and fixes some important notation.
Section 3 reviews classical (monotone) Lagrangian Floer theory, and discusses the difficulties of defining $HF(R,L_c)$ in this context.
In Section 4 we give the combinatorial description of the Floer complex, and in Section 5 we give some applications.
Section 6 contains a review of Floer theory as developed in \cite{fooo}, and finally in Section 7 we prove that our description of the Floer complex given in Section 4 agrees with the one defined in Section 6.

\vspace{.5cm}

\textbf{Acknowledgements}  We would like to thank Miguel Abreu from whom we learned the argument used to prove Theorem \ref{product}.
We would also like to thank Yong-Geun Oh for many helpful discussions on this topic.
We also thank the referees for the extensive and helpful comments.
The second author was supported by FCT-Portugal, research grant SRFH/BD/30381/2006.

\section{Compact Toric Manifolds}

This section contains a brief review of toric manifolds, mainly to establish notation. We follow \cite{toricfooo} closely.

Consider the convex polytope $P$ in $\R^n$ defined by
$$\set{u \in \Mr}{\langle u,v_j \rangle \geq \lambda_j , j=1,...,m},$$
where $\langle \cdot,\cdot \rangle$ is the dot product in $\R^n$ and $v_j \in \R^n$ are the inward normal vectors to the facets of $P$. 
Denote the components of the $v_j$'s as
\begin{equation}\label{eq:components}
v_j=(v^{1}_{j}, \ldots,v^{n}_{j}).
\end{equation}
Let $e_1,\ldots,e_n$ be the standard basis of $\R^n$.
To simplify the notation, we will assume (without loss of generality) that 
\begin{equation}\label{fan}
v_j=e_j,\,\lambda_j=0\textrm{ for $j=1,\ldots,n$.}
\end{equation}

Let $\pi: \Z^m \rightarrow \Z^n$ be the map that sends $e_i$ to $v_i$ for $i=1,\ldots,m$. 
This gives rise to the exact sequence
\begin{equation}\label{exactseq}
0 \rightarrow \mathbb{K}=\Ker(\pi) \rightarrow \Z^m \stackrel{\pi}{\rightarrow} \Z^n \rightarrow 0.
\end{equation}
which induces another exact sequence
\begin{equation}\label{exactseq2}
0 \rightarrow K \rightarrow \R^m/\Z^m \rightarrow \R^n/\Z^n \rightarrow 0.
\end{equation}
Consider $\C^m$ with the usual symplectic form $\frac{\sqrt{-1}}{2}\sum^{m}_{i=1}{dz_i \wedge d\bar{z}_i}$.
The standard action of the real torus $T^m$ on $\C^m$ is Hamiltonian and its moment map is the map
\begin{equation}
\mu(z_1, \ldots,z_m)=\frac{1}{2}\big(|z_1|^2, \ldots ,|z_m|^2 \big).
\end{equation}
Restricting to $K$ gives a Hamiltonian action of $K$ on $\C^m$ with moment map
\begin{equation}\label{muK}
\pi_K(z_1, \ldots,z_m)=\frac{1}{2} \Big( \sum^{n}_{j=1}{v^{j}_{n+1}|z_j|^2}-|z_{n+1}|^2, \ldots ,\sum^{n}_{j=1}{v^{j}_{m}|z_j|^2}-|z_{m}|^2 \Big) \in \R^{m-n},
\end{equation}
where we are identifying $\R^{m-n}$ with the dual of the Lie algebra of $K$ via the following basis of $\mathbb{K}$: 
\begin{equation}\label{Kbasis}
Q_{i-n}=(v^{1}_{i}, \ldots, v^{n}_{i},0, \ldots , -1, \ldots, 0), \ i=n+1, \ldots, m.
\end{equation}
Then for any regular $r \in \pi_K(\C^m)\subset \R^{m-n}$, $$X_P = \pi^{-1}_K(r)/K$$ is a smooth compact manifold that inherits a symplectic form via symplectic reduction.
The following theorem is standard (see \cite{guillemin}):
\begin{theorem} 
\begin{enumerate}
	\item Let $$r_a=(\lambda_{n+1}, \ldots,\lambda_{m})$$
and let $\pi_T$ be the moment map of the residual $T^n=T^m/K$ action on $X_P$, i.e. $\pi_T$ is the map induced from
\begin{equation}\label{muT}
\tilde{\pi}_T(z_1, \ldots,z_m)=\frac{1}{2}\big(|z_1|^2, \ldots ,|z_n|^2 \big), 
\end{equation}
defined on $\pi^{-1}_{K}(r_a)$.

Then the image of $\pi_T$ is $P$, that is
	 $$\pi_T(X_P)=P.$$
	\item $X_P$ is a K\"ahler  manifold with complex structure induced from $\C^m$.
\end{enumerate}
\end{theorem}

Now, recall the notion of homogeneous coordinates:
Points in $X_P$ are equivalence classes of elements of $\pi^{-1}_{K}(r_a)\subset \C^m$, hence can be denoted as
\begin{displaymath}
[z_1,\ldots,z_m]
\end{displaymath}
for complex numbers $z_i$.
The equivalence relation comes from the action of $K$: for any non-zero complex number $z$ with $|z|=1$ and tuple of integers $(k_1,\ldots,k_m)\in\mathbb K$, we have
\begin{displaymath}
[z_1,\ldots,z_m]=[z^{k_1}z_1,\ldots,z^{k_m}z_m].
\end{displaymath}

Finally, let $\tilde{\tau}:\C^m\to\C^m$ be the usual complex conjugation map. $\tilde{\tau}$ descends to an anti-holomorphic involution $\tau:X_P\to X_P$. 
Let 
\begin{displaymath}
R_P=\set{x\in X_P}{\tau(x)=x}.
\end{displaymath}
$R_P$ is the real Lagrangian in $X_P$.
 For each $c \in \textrm{int}(P)$, let 
\begin{displaymath}
L_c=\pi^{-1}_{T}(c)
\end{displaymath}
 be the Lagrangian torus fiber over $c$. Note that $L_c$ is preserved by the map $\tau$ defined above.
The Lagrangians we will study are $R_P$ and $L_c$.

\section{Classical Floer cohomology}

In this section we will briefly review classical (monotone) Lagrangian Floer cohomology and then discuss how the Floer cohomology of the pair $(R_P,L_c)$ fits into this context.
It turns out that if the classical definition of Floer cohomology is used, then the geometry of the pair of Lagrangians $(R_P,L_c)$ can be exploited to admit a particularly nice combinatorial description of the cochain complex, just as in \cite{alston}.
This description will be given in the next section.
However, since $L_c$ is not monotone, standard methods cannot be used to show that the cohomology of the complex is invariant under the choice of almost complex structure.
The only way around this problem seems to be to apply the Lagrangian Floer theory machinery developed in \cite{fooo}.
The details of this analysis will be relegated to Sections 6 and 7.

We turn now to a quick review of classical Floer cohomology.
Let $(M,\omega)$ be a symplectic manifold, $J$ a compatible almost complex structure, and $L_0$ and $L_1$ two transversely intersecting Lagrangian submanifolds.
The Floer cochain complex is 
\begin{displaymath}
C(L_0,L_1)=\bigoplus_{p\in L_0\cap L_1}\Z_2\cdot p.
\end{displaymath}
Given $p,q\in L_0\cap L_1$, let 
\begin{displaymath}
\mathcal M(p,q:\mu=1)
\end{displaymath}
denote the set of $J$-holomorphic strips $u:\R\times[0,1]\to M$ of Maslov index 1 that connect $p$ to $q$ and have bottom boundary on $L_0$ and top boundary on $L_1$.
The Floer differential
\begin{displaymath}
\delta:C(L_0,L_1)\to C(L_0,L_1)
\end{displaymath}
is given by the formula
\begin{displaymath}
\delta(p)=\sum_{q\in L_0\cap L_1}n(p,q)q,
\end{displaymath}
where $n(p,q)$ is the mod 2 number of components of $\mathcal M(p,q:\mu=1)$ (that is, $n(p,q)$ is the number of isolated holomorphic strips that connect $p$ to $q$).

The Floer cohomology $HF(L_0,L_1)$ of the pair $(L_0,L_1)$ is then the cohomology of the complex $(C(L_0,L_1),\delta)$.
Of course, for this definition to be sensible several things need to be shown:
\begin{enumerate}
\item $\delta$ is well-defined, i.e. $n(p,q)$ is finite for all $p$ and $q$,
\item $CF(L_0,L_1)$ is a cochain complex, i.e. $\delta^2=0$,
\item the cohomology does not depend on the choice of the almost complex structure $J$, and
\item the cohomology is invariant under Hamiltonian diffeomorphisms.
\end{enumerate}
These properties can be interpreted as compactness properties of the moduli space of strips:
Item 1 is equivalent to the compactness of $\mathcal M(p,q:\mu=1)$, item 2 is equivalent to the compactness of $\mathcal M(p,q:\mu=2)$ modulo broken trajectories, and items 3 and 4 are equivalent to the compactness of a parameterized moduli space modulo broken trajectories.

Non-compactness of the moduli spaces is a result of disc bubbling; the standard way to ensure compactness is to impose topological conditions that minimize the disc bubbling that can occur.
The two standard topological notions that come into play are monotonicity and the minimal Maslov number:
Let $\mu:\pi_2(M,L_i)\to\Z$ denote the Maslov index homomorphism.
Then $L_i$ is called monotone if there exists a positive constant $c$ such $\mu(u)=c\int u^*\omega$ for all $u\in\pi_2(M,L_i)$, and the minimal Maslov number $\Sigma_{L_i}$ of $L_i$ is the positive generator of the image of $\mu$.

In \cite{oh}, it is shown that if the Lagrangian submanifolds are monotone, and $\Sigma_{L_i}\geq 3$ for $i=1,2$, then the Floer cohomology is well-defined and does not depend on the choice of a generic almost complex structure.
The requirement $\Sigma_{L_i}\geq 3$ can be relaxed to $\Sigma_{L_i}\geq 2$ if an additional property is satisfied.
The additional property is the following:
Let $\mathcal M_1(L_i:\mu=2)$ denote the moduli space of $J$-holomorphic discs of Maslov index 2 with one marked (boundary) point and boundary lying on $L_i$.
The fact that $L_i$ is monotone implies that $\mathcal M_1(L_i:\mu=2)$ is a closed manifold of dimension $\dim L_i$; thus the mod 2 intersection number $$o(L_i):=[\mathcal M_1(L_i:\mu=2)]\cdot [p] \textrm{ mod }2$$ is well-defined and does not depend on the choice of the generic point $p\in L_i$.
The additional property needed is $o(L_1)=o(L_2)$.
That is, if $o(L_1)=o(L_2)$ then the Floer cohomology of the pair $(L_0,L_1)$ is well-defined and does not depend on the choice of generic $J$.

Note that the difficulty in the case $\Sigma_{L_i}=2$ arises only in trying to prove $\delta^2=0$.
The reason is that the moduli space $\mathcal M(p,q:\mu=2)$ can be non-compact due to disc bubbling; in fact it is straightforward to check that 
$$\delta^2(p)=(o(L_1)+o(L_2))\cdot p$$
(see \cite{alston} for a detailed explanation of this).
Thus the condition $o(L_1)=o(L_2)$ is needed to get $\delta^2=0$.

\bigskip

Now that we have outlined the general Floer theory of monotone Lagrangian submanifolds, we will examine how the torus fibers $L_c$ and the real Lagrangian $R_P$ fit into this framework.
The first thing to notice is that in general $L_c$ and $R_P$ are not monotone.
This fact is mitigated to some extent by our assumption that the ambient toric manifold $X_P$ is Fano.
With the standard complex structure $J$ we get a positivity property (weaker than monotonicity), namely holomorphic discs with boundary on $L_c$ have Maslov index at least 2, and holomorphic discs with boundary on $R_P$ have positive Maslov index (also at least 2 if $R_P$ is orientable, assume this is the case for now).

We can mimic the construction of Floer cohomology given above for the pair $(R_P,L_c)$ and then see what properties it has.
The first problem encountered is that the numbers $n(p,q)$ may not be finite (actually they are, but this is not a priori true) so the coefficient ring should be replaced by a Novikov ring, say
$$ \unov= \Big\{ \sum^{\infty}_{i=1}{a_iT^{\lambda_i}} \Big| \ a_i \in \Z_2, \lambda_i \in \R, \  \lambda_i \rightarrow \infty \Big\}. $$
Then the cochain complex is
\begin{displaymath}
C(R_P,L_c)=\bigoplus_{p\in L_0\cap L_1}\unov\cdot p
\end{displaymath}
and the map $\delta:C(R_P,L_c)\to C(R_P,L_c)$ is given by
\begin{displaymath}
\delta(p)=\sum_{q\in L_0\cap L_1}\sum_{E>0}\#(\mathcal M(p,q:\mu=1:\omega=E))T^Eq,
\end{displaymath}
where $\omega=E$ denotes that the symplectic area of the strip is $E$.
If the strips are regular (that is, the linearized Cauchy-Riemann operator is surjective), then $\delta$ is well-defined.

The next thing to consider is whether or not $\delta^2=0$.
Examination of the proofs in the monotone case shows that the obstruction again comes only from Maslov index 2 discs. 
However, since the coefficient ring is the Novikov ring, the obstructions $o(L_c)$, $o(R_P)$ need to be weighted with the symplectic area.
That is,
\begin{displaymath}
o(L_c):=\sum_{E>0} ([\mathcal M_1(L_c:\mu=2:\omega=E)]\cdot[p])T^E
\end{displaymath}
where $p\in L_c$ is a generic point.
A similar definition holds for $R_P$.
Assume now that Maslov index 2 discs are regular, and the evaluation map is transverse to $L_c\cap R_P$.
Then if $o(L_c)=o(R_P)$ and Maslov index 2 strips are regular, it follows that $\delta^2=0$.
Thus the cohomology of the complex is well-defined.

However, problems arise when trying to show that the cohomology does not depend on $J$ and is invariant under Hamiltonian diffeomorphism.
Consider invariance on $J$ first.
For a generic $J$ it is not necessarily true that $\delta^2=0$ because Maslov index $0$ discs may exist, and hence the proof given above that $\delta^2=0$ does not work.
Moreover, even if $\delta^2=0$, there is no reason that the cohomology has to be the same, because the parameterized moduli spaces used to construct chain maps and chain homotopies can also exhibit Maslov index $0$ disc bubbling.
(This explains why the positivity we get for the standard complex structure is weaker than monotonicity--it is not a topological positivity property, but rather a positivity property that depends on the discs that exists for the standard $J$.)
Trying to prove invariance under Hamiltonian diffeomorphism runs into the same problem:
Once you move one of the Lagrangians, it is possible for Maslov index 0 discs to exist, and the proofs break down.

There does not seem to be an easy way to deal with these problems other than to use the full Floer theory machinery developed in \cite{fooo}.
Therefore, our approach to the Floer theory of $(R_P,L_c)$ is the following:
We will use the definition of Floer cohomology given in \cite{fooo}, and appeal to the theory therein to claim that the Floer cohomology is well-defined, and is invariant under Hamiltonian deformation.
Actually, since we will use $\Z_2$ coefficients, we have to appeal to the theory in \cite{fooo8}, which is the extension of \cite{fooo} to the $\Z$ coefficient case (which holds for $\Z_2$ as well).
The dependence on the complex structure $J$ then becomes a more delicate matter.
To use $\Z$ coefficients we have to avoid non-trivial isotropy in the Kuranishi structures on the moduli spaces, and to do this we need to restrict to the class of spherically positive complex structures.
The Floer cohomology then depends only on the components of the set of all spherically positive complex structures (i.e. if $J_1$ and $J_2$ are in the same component then they have the same Floer cohomology).
The standard complex structure is spherically positive, so this restriction does not cause any problems for us.
Using the full Floer machinery also allows us to drop the assumption that every holomorphic disc with boundary on $R_P$ has Maslov index at least 2.

The Floer cochain complex defined in \cite{fooo} depends on the choice of weak bounding cochains for the Lagrangian submanifolds.
We will show that for an appropriate choice of weak bounding cochains and for the standard complex structure this Floer complex is actually the complex we want it to be, namely the one defined above where we just need to count Maslov index 1 strips that connect points in $R_P\cap L_c$.
These strips will then be counted following the ideas in \cite{alston} and used to give a combinatorial description of the Floer complex and differential.
This description forms the basis of the applications we will give.

The discussion on the choice of bounding cochains and the proof that the complexes are the same are relegated to Section 7. 
Until then we will omit the bounding cochains from the notation.
The proof of the equality of the complexes will rely on several facts (in addition to general properties of $A_\infty$-algebras and bounding cochains):
Maslov index 1 
strips are regular, Maslov index 2 discs with boundary on $L_c$ are regular and the evaluation map is transverse to $R_P\cap L_c$, and $R_P$ is unobstructed.
We will discuss the first three items in this section and the next, because they also appear in the classical monotone setting.
The unobstructedness of $R_P$ is rather technical and will be discussed in Section 7.
There it will be shown also that $o(R_P)=0$; intuitively the reason is that discs with boundary on $R_P$ can be conjugated using the Schwarz reflection principle (since $R_P$ is the set of real points), and hence the number of discs is even, i.e. $0$ mod 2.

For the applications of Floer cohomology given in this paper, we will need to twist the Floer complex by a locally constant sheaf (this is analogous to twisting Floer cohomology defined over $\C$ by a flat line bundle, see \cite{cho}) and take the coefficient ring to be the Novikov ring over the algebraic closure of $\Z_2$.
Here are the details of this construction:
Let $\mathbb{F}_2$ be the algebraic closure of $\Z_2$ and let $\unovf$ be the Novikov ring over $\mathbb{F}_2$,
\begin{displaymath}
 \unovf= \Big\{ \sum^{\infty}_{i=1}{a_iT^{\lambda_i}} \Big| \ a_i \in \mathbb{F}_2, \lambda_i \in \R, \  \lambda_i \rightarrow \infty \Big\}.
\end{displaymath}
Let $\rho:\pi_1(L_c)\to \F_2^*$ be a homomorphism and consider the sheaf of locally constant sections of the $\F_2$-line bundle
\begin{displaymath}
\mathcal L_\rho=(\tilde L\times \mathbb F_2)/(x\cdot\gamma,v)\sim (x,\rho(\gamma)v)
\end{displaymath}
over $L_c$, where $\tilde L$ is the universal cover of $L$.
The vector bundle structure of $\mathcal L_\rho$ is
\begin{displaymath}
[x,v]+[x,w]=[x,v+w],\ \lambda[x,v]=[x,\lambda v].
\end{displaymath}
The twisted cochain complex is
\begin{displaymath}
C(R_P,(L_c,\mathcal L_\rho);\unovf)=\bigoplus_{p\in R_P\cap L_c}\textrm{Hom}(\mathcal L_\rho|_p,\F_2)\otimes \unovf.
\end{displaymath}
The differential $\delta$ is the same as the previously defined with the exception that the holonomy of sections around holomorphic strips needs to be taken into account.
More precisely:
Given a strip $u\in \mathcal M(p,q:\mu=1:\omega=E)$, let $par(l_u):\mathcal L_{\rho}|q\to\mathcal L_{\rho}|p$ denote parallel translation along the top boundary of $u$ (the fibers of $\mathcal L_\rho$ are discrete so a connection is not needed to define parallel translation).
Then for $\alpha\in \textrm{Hom}(\mathcal L_\rho|p,\F_2)$
\begin{displaymath}
\delta(\alpha)=\sum_{q\in R_P\cap L_c}\sum_{u\in\mathcal M(p,q:\mu=1)}(\alpha\circ par(l_u))\otimes T^{\omega(u)}.
\end{displaymath}
The well-definedness and invariance of twisted Floer cohomology mirrors the discussion given above for non-twisted Floer cohomology.
One difference, however, is that the obstruction $o(L_c)$ needs to include coefficients coming from the holonomy of $\mathcal L_\rho$.
Thus the correct definition of the obstruction is
\begin{displaymath}
o(L_c,\rho):=\sum_{\beta\in\pi_2(X,L_c),\mu(\beta)=2} ([\mathcal M_1(L_c,\beta)]\cdot[p])\textrm{hol}_{\mathcal L_\rho}(\partial\beta)T^{\omega(\beta)}.
\end{displaymath}
Note that $\mathcal L_\rho$ is a line bundle with fibers isomorphic to $\F_2$; hence $\textrm{hol}_{\mathcal L_\rho}(\cdot)$ can canonically be identified with an element of $\F_2$.
The upshot is that if $o(L_c,\rho)=0\in \unovf$ then $\delta^2=0$, i.e. the Floer cohomology of the pair $(R_P,(L_c,\mathcal L_\rho))$ is defined.

\bigskip

In the remainder of this section we will study holomorphic discs with boundary on $L_c$.
The purpose of this is two-fold: one, we need to write down a formula for $o(L_c)$, and two, the knowledge will be needed in the next section to write down a formula for $\delta$.

Recall that $P$ is the defining polytope of the toric manifold $X_P$, as in Section 2.
Let $ \partial P = \cup^{m}_{i=1} \partial_i P $ be the decomposition of the boundary of $P$ into facets. 
Let $\beta_i \in H_2(X_P, L_c)$ be the element such that
\begin{equation}
\beta_i \cap [\pi^{-1}_{T}(\partial_j P)]= \left\{ 
\begin{array}{ll}
1 , & i=j \\
0 , & i \neq j
\end{array} 
\right..
\end{equation}
It follows from Theorem 5.1 in \cite{chooh} that $\mu(\beta_i)=2$.
The following theorem was proved by Cho and Oh in \cite{chooh} under some technical assumptions which were later removed in Section 11 of \cite{toricfooo}.
\begin{theorem} \label{discs}
Let $X_P$ be a Fano toric manifold and $L_c$ a fiber of the moment map. Then:
\begin{enumerate}
	\item If $\mu(\beta) \leq 0$ and $\beta \neq 0$, then $\mathcal{M}_1(L_c, \beta)^{\bf s}= \emptyset$. Here \textbf{s} is a section of the Kuranishi structure on this moduli space (see Section 7 for further explanation).
	\item For $i=1,\ldots,m$, $\mathcal{M}_1(L_c, \beta_i)$ does not have boundary. 
	\item $ev_0:\mathcal{M}_1(L_c, \beta_i) \rightarrow L_c$ is a diffeomorphism.
	\item If $\mu(\beta)=2$ and $\beta \neq \beta_i$ for some $i$, then $\mathcal{M}_1(L_c, \beta_i)= \emptyset$.
\end{enumerate}
\end{theorem}

Cho and Oh \cite{chooh} also provide a complete description of discs of Maslov index 2:
\begin{theorem}
Let $u:(D^2,\partial D^2)\rightarrow(X_P,L_c)$ be a holomorphic map such that $[u]=\beta_i$. Up to reparameterization, $u$ can be written in homogeneous coordinates as
\begin{equation} \label{top}
u(z)=[c_1,\ldots,zc_i,\ldots,c_m]
\end{equation}
where $(c_1,\ldots,c_m) \in \C^m$ are such that $\pi_T(c_1,\ldots,c_m)=c$.
Moreover the moduli space $\mathcal{M}_1(L_c, \beta_i)$ is regular. 
\end{theorem} 
We will also need the following result from \cite{chooh}.
\begin{theorem}
The area of a disc in the class $\beta_i$ is
$$E_i := 2\pi ( \langle c, v_i \rangle - \lambda_i).$$
\end{theorem} 

We will now describe the obstruction $o(L_c)$.
First, we need some notation to describe the sheaf that will be used to twist the Floer complex.
Let
\begin{eqnarray}\label{sheaf}
l_i : S^1 & \rightarrow  & L_c \nonumber \\ 
      e^{i\theta} & \rightarrow & [c_1,\ldots, e^{i\theta}c_i,\ldots,c_m]
\end{eqnarray}
(for $i=1,\ldots,n$) be a basis of $H_1(L_c, \Z)$. 
As before, let $\mathcal{L}_{\rho}$ be the sheaf given by a homomorphism $\rho : H_1(L_c, \Z) \rightarrow (\mathbb{F}_2)^{\ast}$. 
Let $\rho_i = \rho(l_i)$ and $\rho^{v_j}=\rho^{v^1_j}_{1} \ldots \rho^{v^n_j}_{n}$ (see (\ref{eq:components}) for the definition of $v_j$).

\begin{proposition}
Let $L_c$ be a torus fiber in $X_P$.
Then the obstruction $o(L_c)$ is 
\begin{equation}\label{m0}
o(L_c) = \sum^{m}_{i=1} T^{E_i}.
\end{equation}
If the Floer complex is twisted by the sheaf $\mathcal L_\rho$, the obstruction is
\begin{equation}\label{m0f}
o(L_c, \rho) = \sum^{m}_{j=1} \rho^{v_j}T^{E_j}.
\end{equation}
\end{proposition}
\begin{proof}
By Theorem 3.1, the only homotopy classes that admit holomorphic discs are $\beta_1,\cdots,\beta_m$.
These discs are regular and the evaluation map $ev_1:\mathcal M(L_c,\beta_i)\to L_c$ is a diffeomorphism, i.e. each point of $L_c$ intersects the boundary of one disc of class $\beta_i$.
Finally, by Theorem 3.3, $\omega(\beta_i)=E_i$.
Therefore
\begin{displaymath}
o(L_c) = \sum^{m}_{i=1} T^{E_i}.
\end{displaymath}
The proof of the second statement follows by also taking into consideration the holonomy of $\mathcal L_\rho$ and using the fact that the boundary of the disc
\begin{displaymath}
z\mapsto [c_1,\ldots,c_kz,\ldots,c_m], \ |z|=1
\end{displaymath}
can also be written as
\begin{displaymath}
z\mapsto [c_1z^{v_k^1},\ldots,c_nz^{v_k^n},c_{n+1},\ldots,c_m]
\end{displaymath}
by using the action of the subgroup of $K$ generated by $Q_{k-n}$ (defined in (\ref{Kbasis})) when $k \geq n$. 
\end{proof}

\section{Floer Cohomology of $(R_P , L_c)$}

Let $X_P$ be a Fano toric manifold (as in Section 2), let $R=\textrm{Fix}(\tau)$ be the real Lagrangian submanifold ($\tau:X_P\to X_P$ is complex conjugation) and let $L_c = \pi^{-1}_{T}(c)$ for $c \in \textrm{int}(P)$ be a Lagrangian torus fiber.
Let $\mathcal{L}_{\rho}$ be a locally constant sheaf on $L_c$.
In this section we will describe the complex 
$$(C(R_P,(L_c,\mathcal L_\rho);\unovf), \delta)$$
as defined in the previous section.
Namely $\delta$ wil be the classical Floer differential given by counting holomorphic strips (with the appropriate weights) of Maslov index 1. 
In Section 7 we will prove that with the standard toric complex structure, the definition of the Floer complex given in \cite{fooo} reduces to the one above. 
It will then follow from the theory in \cite{fooo} and \cite{fooo8} that the Floer cohomology is invariant under Hamiltonian deformation and does not depend on the choice of spherically positive almost complex structure in the same connected component (of spherically positive almost complex structures) as the toric complex structure.

$L_c$ and $R$ intersect transversely in $2^n$ points. If $p\in L_c\cap R$ we can write (in a non-unique way) in homogeneous coordinates 
$p=[c_1,\ldots, c_m]$
for some $ c_i \in \R^{\ast}$ satisfying $\pi_{T}[c_1,\ldots, c_m]=c$ and $\pi_{K}(c_1,\ldots, c_m)=r$. From the expressions (\ref{muK}) and (\ref{muT}) we see that the last two conditions determine the norms of the $c_i$'s, but not the signs.
Using the action of $K$ we can normalize the choices of the $c_i$'s. For each $i=n+1, \ldots, m$ we use the one-parameter subgroup of $K$ generated by $Q_{i-n}$ in (\ref{Kbasis}) to make $c_i$ positive. In this way we can write each $ \ p \in L_c \cap R$ uniquely as 
\begin{equation}\label{int}
p=[(-1)^{\epsilon_1} c'_1,\ldots, (-1)^{\epsilon_n} c'_n, c'_{n+1},\ldots,c'_m]
\end{equation}
where $\epsilon_i \in \ \{0,1\}$ and the $c'_i$'s are determined by the conditions
\begin{itemize}
	\item $(c'_1,\ldots, c'_m) \in \pi^{-1}_{K}(r)$,
	\item $[c'_1,\ldots, c'_m] \in \pi^{-1}_{T}(c)$,
	\item $c'_i$ is positive for all $i$.
\end{itemize}
Therefore we can identify the intersection points with tuples $\epsilon=(\epsilon_1,\ldots, \epsilon_n)$ and in this way we get
\begin{equation} \label{basis}
C(R,L_c; \unovf)= \bigoplus_{\epsilon} \epsilon \cdot \textrm{Hom}(\mathcal{L}|_{\epsilon}, \mathbb{F}_2) \otimes \unovf  \ \ \textrm{where} \ \epsilon=(\epsilon_1,\ldots, \epsilon_n) \in \{0,1\}^n.
\end{equation}

To calculate the Floer differential we need to describe the holomorphic strips with boundaries on the Lagrangians. We will reduce this problem to the case of discs with boundary on the torus $L_c$.
Let $u: \R \times [0,1] \rightarrow X_P$ be a holomorphic strip with the bottom boundary on $R_P$ and top boundary on $L_c$.
Noting that $\tau$ preserves $L_c$ we can use the Schwarz reflection principle to reflect $u$ about $R$ to obtain a map
$$\tilde u:\mathbb R\times[-1,1]\to X$$
with the properties
\begin{itemize}
\item $\tilde u|\mathbb R\times[0,1]=u$,
\item the energy of $\tilde u$ is twice that of $u$, and
\item both the top and bottom boundaries of $\tilde u$ lie on $L_c$.
\end{itemize}
Note that $\mathbb R\times[-1,1]$ is conformally equivalent to $D^2\setminus\{-1,1\}$, so the domain of $\tilde u$ can be thought of as $D^2\setminus\{-1,1\}$.
Then, by the removable singularities theorem, $\tilde u$ extends to a holomorphic map $\tilde u:(D^2,\partial D^2)\to(X,L_c)$.
The Maslov index of $\tilde u$ (thinking of $\tilde u$ as a disc) is twice that of $u$. 
In conjunction with Theorem 3.2 this proves
\begin{proposition}\label{strips}
Let  $p=[c_1,\ldots, c_m] \in L_c \cap R$. There are $m$ ($m$ is the number of facets of $P$) holomorphic strips of Maslov index 1 that start at $p$. They are the top halves of the discs
$$u_j:z\mapsto[c_1,\ldots,-c_jz,\ldots,c_m].$$
\end{proposition}

Now that we have an explicit description of all the strips, we need to verify that they are regular:
\begin{lemma}
Let $u$ be a holomorphic strip with $\mu(u)=1$.
Then $D_u\bar\partial$ is surjective.
\end{lemma}
\begin{proof}
By the previous proposition, $u$ is the top half of one of the discs $u_j$.
Without loss of generality assume $u_j=u_1$.
The image of $u_1$ is contained in the affine open set corresponding to the top dimensional cone spanned by $v_1,\ldots,v_n$ in the fan $\Sigma_P$.
Using this affine set, we can choose (holomorphic) coordinates such that the map $u_1$ is
\begin{displaymath}
u_1:(D^2,\partial D^2)\to(\mathbb C^n,(S^1)^n),\  z\mapsto (z,1,\ldots,1).
\end{displaymath}
The strip $u$ is the top half of $u_1$, so $u=v_1\times\cdots\times v_n$ where $v_i:\mathbb R\times[0,1]\to \mathbb C$, $v_1$ is the standard biholomorphism from the strip onto the top half of the unit disc, and $v_i=1$ for $i\geq 2$.
$X_P$ is K\"ahler so the linearization $D_u\bar\partial$ is the Dolbeault operator.
Therefore $D_u\bar\partial=D_{v_1}\bar\partial\oplus\cdots\oplus D_{v_n}\bar\partial$.
All of the Maslov indices $\mu(v_i)$ are nonnegative, so each $D_{v_i}\bar\partial$ is surjective (see for example \cite{alston}).
Therefore $D_u\bar\partial$ is surjective.
\end{proof}

We can now describe the Floer differential. 
By Proposition \ref{strips}, if $u_j\in\mathcal M(p,q : \mu=1)$ and $$ p=[c_1,\ldots,c_j,\ldots,c_m]$$
then
$$ q=[c_1,\ldots,-c_j,\ldots,c_m].$$ 
Using the action of $K$, $q$ can be rewritten as $$q=[(-1)^{v^1_j}c_1,\ldots, (-1)^{v^n_j}c_n, c_{n+1},\ldots, c_m].$$ 
where the $v^i_j$ are as in (\ref{eq:components}). 
Therefore, in terms of the basis (\ref{basis}), the strip $u_j$ contributes to the differential as the map
$$(\epsilon_1,\ldots,\epsilon_n)\mapsto(\epsilon_1 + v^{1}_{j},\ldots,\epsilon_n + v^{n}_{j}).$$

To complete the description of the differential we need to include the contribution of the locally constant sheaf $\mathcal{L}_{\rho}$.
For this purpose we will choose a generator of $\textrm{Hom}(\mathcal{L}|_{\epsilon},\mathbb{F}_2)$ for each $\epsilon=(\epsilon_1,...,\epsilon_n)$. For $\epsilon=(0,\ldots,0)$ we pick $\alpha_0$, an arbitrary non-zero element of $\textrm{Hom}(\mathcal{L}|_{0},\mathbb{F}_2)$.
Given other $\epsilon$ consider the path 
$$l_\epsilon(t):=[e^{it\epsilon_1}c'_1,\ldots,e^{it\epsilon_n}c'_n, c'_{n+1},\dots, c'_m] \in L_c, \ t \in [0,\pi],$$
connecting $(0,\ldots,0)$ and $\epsilon=(\epsilon_1,...,\epsilon_n)$.
Then, we define
$$\alpha_\epsilon = \rho^{\epsilon/2}\alpha_0 \circ par(l^{-1}_{\epsilon}) \in \textrm{Hom}(\mathcal{L}|_{\epsilon},\mathbb{F}_2).$$
Here we use the notation $\rho^{\epsilon/2}=\rho^{\epsilon_1/2}_1 \ldots \rho^{\epsilon_n/2}_{n}$, where $\rho_i$ are as in (\ref{sheaf}).
Recall, from the definition of the differential, that the strip $u_j$ starting at $\epsilon$ and ending at $\delta$ sends $\alpha_\epsilon$ to $\alpha_\epsilon \circ par(l_j)$.
Here $l_j$ is the top boundary of the strip $u_j$:
$$[ e^{it v^1_j}(-1)^{\delta_1}c'_1,\ldots,e^{it v^n_j}(-1)^{\delta_n}c'_n,c'_{n+1},\dots, c'_m ], \ t \in [0,\pi].$$
Then we consider the loop $l$ based at $[c'_1,\ldots,c'_m]$ obtained as the concatenation $l^{-1}_{\delta} \circ l^{-1}_j \circ l_{\epsilon}$. Where $l^{-1}_\delta$ is the path $l_{\delta}$ with the opposite orientation. 

By definition, $\alpha_0 \circ par(l) = \rho(l) \alpha_0$. Now, one can easily check that 
$$\rho(l)= \rho^{(\epsilon-\delta-v_j)/2}.$$
Observe that $\delta= \epsilon + v_j \ (mod \ 2)$ and so the exponent is an integer.
This implies
\begin{eqnarray*}
\rho^{\epsilon/2} \alpha_0\circ par(l^{-1}_{\epsilon}) & = & \rho^{(\delta + v_j)/2} \alpha_0 \circ par(l^{-1}_{\delta}) \circ par(l^{-1}_j), \\
\alpha_\epsilon \circ par(l_j) & = & \rho^{v_j/2} \alpha_{\delta}.
\end{eqnarray*}
With this at hand we can state the main theorem of the paper:
\begin{theorem} \label{main}
Let $X_P$ be a Fano toric manifold with Lagrangian submanifolds $R=\textrm{Fix}(\tau)$ and $L_c$. Equip the torus fiber with a locally constant sheaf $ \mathcal{L}_{\rho}$. Suppose $o(L_c, \rho)=0$. Then the Floer cohomology of the pair $HF(R,(L_c,\mathcal{L}_{\rho});\unovf)$ is well defined and isomorphic to the cohomology of the complex $(C,\delta)$ given by
\begin{eqnarray}\label{formula}
C & =& \bigoplus_{\epsilon} \epsilon \cdot \unovf ,  \ \epsilon=(\epsilon_1,\ldots, \epsilon_n) \in (\Z_2)^n \nonumber \\
\delta & = & \sum^{m}_{j=1}{\sqrt{\rho^{v_j}}T^{E_j/2} f_j} \\
f_j(\epsilon_1,\ldots,\epsilon_n) & = & (\epsilon_1 + v^{1}_{j},\ldots,\epsilon_n + v^{n}_{j}). \nonumber
\end{eqnarray}
\end{theorem}
\begin {proof}
The description of the complex $(C, \delta)$ follows from the discussion above.

Observing that the $f_i$'s commute and $f^2_i=id$ we compute
$$\delta^2=\sum^m_{i,j=1} \sqrt{\rho^{v_i}}\sqrt{\rho^{v_j}}T^{\frac{E_i+E_j}{2}}f_i f_j = \sum^{m}_{j=1} \rho^{v_j}T^{E_j} id = o(L_c, \rho) id$$ 
And so the statement about the obstruction follows.
\end{proof}

\section{Applications}
Theorem \ref{main} reduces the problem of calculating the Floer cohomology of the pair $(R,L_c)$ to a purely computational problem. However, there does not seem to be an easy way to read off the rank of the cohomology from the combinatorial data defining the toric manifold. 

Our first application is to show that, whenever possible, the cohomology does not vanish.
Let us describe what we mean by ``whenever possible'':
It is a general fact that $HF(L)$ is a ring and $HF(L,L')$ is a $HF(L)-HF(L')$ bimodule. 
Also the fundamental class $[L] \in C_{n}(L)$ is the identity in $HF(L)$ and acts as a unit in $HF(L,L')$ (see Section 3.7 in \cite{fooo}). 
The same remark holds for $HF(L')$.
Thus if $HF(L)$ or $HF(L')$ is $0$, then $HF(L,L')$ is necessarily $0$.

For each $i \in {1,\ldots,n}$ define
\begin{equation}\label{zed}
Z_i= \sum^{m}_{j=1} v^i_j \rho^{v_j} T^{E_j}
\end{equation}
and let $Z=(Z_1,\ldots,Z_n)$.
We will see in Section 7 that:
\begin{itemize}
\item $HF(R ;\unovf) = H^{\ast}(R, \Z_2) \otimes \unovf \neq0$, and
\item $HF(L_c,\mathcal{L}_{\rho};\unovf)=$
$\left\{
\begin{array}{ll}
H^{\ast}(L_c, \Z_2) \otimes \unovf  & \textrm{if} \ Z=0 , \\
0  & \textrm{otherwise.} 
\end{array} 
\right.$
\end{itemize}
Therefore if $Z \neq 0$ then $HF(R,(L_c,\mathcal{L}_{\rho}))=0$.
We will now prove the converse.
\begin{theorem}\label{nonzero}
If $Z=0$, i.e. $HF(L_c,\mathcal{L}_{\rho};\unovf)\neq 0$, then $HF(R,(L_c,\mathcal{L}_{\rho});\unovf) \neq 0$.
\end{theorem}
\begin{proof}
We will prove, by induction on $n$, the non-vanishing of the cohomology of any complex of the form
\begin{eqnarray}\label{gencomplex}
C_n & =& \bigoplus_{\epsilon} \epsilon \cdot \unovf ,  \ \epsilon=(\epsilon_1,\ldots, \epsilon_n) \in (\Z_2)^n \nonumber \\
d_n & = & \sum^{m}_{j=1}{a_j T^{E_j/2} f_j}  \ \ , a_j \in \mathbb{F}_2 \\
f_j(\epsilon_1,\ldots,\epsilon_n) & = & (\epsilon_1 + v^{1}_{j},\ldots,\epsilon_n + v^{n}_{j}). \nonumber
\end{eqnarray}
satisfying
\begin{align*}
\quad \sum^{m}_{j=1} a^2_j T^{E_j}&=0 &  \textrm{Condition 1}\\
\quad \sum^{m}_{j=1} v^i_j a^2_j T^{E_j}&=0, \ \textrm{for all} \ i. & \textrm{Condition 2}
\end{align*}
The complex (\ref{formula}) is a special case of this, but we need to consider this slightly more general complex for the induction argument. In the special case of (\ref{formula}), Condition 1 is equivalent to $o(L_c, \rho)=0$ and Condition 2 is just $Z=0$. 

Let us check the case $n=1$.
By reindexing, if necessary, we can assume that 
\begin{equation}
v^1_j \equiv 
\left\{
\begin{array}{ll}
0 \ \ (mod \ 2) & \textrm{for} \ j \leq k ,\\ 
1 \ \ (mod \ 2) & \textrm{for} \ j \geq k+1.
\end{array}
\right.
\end{equation}
In the basis $\langle (0),(1)\rangle$, $d$ has the form
\begin{equation}
\left(
\begin{array}{cc}
\sum^{k}_{j=1} a_j T^{E_j/2} & \sum^{m}_{j=k+1} a_j T^{E_j/2} \\
\sum^{m}_{j=k+1} a_j T^{E_j/2} & \sum^{k}_{j=1} a_j T^{E_j/2}
\end{array}
\right).
\end{equation}
 Condition 2 implies $\sum^{m}_{j=k+1} a_j^2 T^{E_j}=0$. Combining with condition 1, we have $\sum^{k}_{j=1} a_j^2 T^{E_j}=0$. Since we are working in characteristic 2, this is equivalent to 
$$\sum^{m}_{j=k+1} a_j T^{E_j/2} \ = \ 0 \ = \ \sum^{k}_{j=1} a_j T^{E_j/2}.$$
So we have $d=0$ and $H^{\ast}(C,d)\neq0$.

Now let us prove the inductive step.
As above, assume that
\begin{equation}
v^{n+1}_j \equiv 
\left\{
\begin{array}{ll}
0 \ \ (mod \ 2) & \textrm{for} \ j \leq k ,\\ 
1 \ \ (mod \ 2) & \textrm{for} \ j \geq k+1.
\end{array}
\right.
\end{equation}
Let $X_0= \sum^{k}_{j=1} a_j T^{E_j/2}f_j$ and $X_1=\sum^{m}_{j=k+1} a_j T^{E_j/2}f_j$.
Noting that the $f_j$'s commute and $(f_i)^2=id$, we compute
$$(X_1)^2 = \sum^{m}_{j=k+1} a_j^2 T^{E_j} \ id$$   
Now, condition 2 for $i=n+1$ implies $\sum^{m}_{j=k+1} a_j^2 T^{E_j}=0$. So, we conclude $(X_1)^2=0$.
As in the case $n=1$, condition 1 implies $\sum^{k}_{j=1} a_j T^{E_j/2} = 0$. This gives $(X_0)^2 =0$.

Now, we consider the splitting $C_{n+1}=B_0 \oplus B_1$, where $B_i= \bigoplus (\epsilon_1,\ldots,\epsilon_n,i)$. Note that $X_0$ preserves the subspaces and $X_1$ interchanges them. Therefore, we can write
\begin{equation}
d=\left(
\begin{array}{cc}
X_0 & X_1 \\
X_1 & X_0
\end{array}
\right).
\end{equation}
We have natural identifications $B_0 \cong B_1 \cong C_{n}$. If we consider the obvious restrictions we have that $(C_{n},d_n=X_0+X_1)$ is a complex, since $X_0$ and $X_1$ are differentials and commute. Obviously, it's of the form (\ref{gencomplex}) and it satisfies the inductive hypothesis, so its cohomology doesn't vanish.

We will now prove the claim by contradiction.
Assume that the cohomology of $(C_{n+1},d_{n+1})$ vanishes.
By induction we can choose $c \in C_{n}$ representing a non-trivial cohomology class. Then $d_n(c)=X_0(c)+X_1(c)=0$. This implies $d_{n+1}(c,c)=0$, where $(c,c)=(c,0) \oplus(c,1)\in B_0 \oplus B_1$.
By assumption, there exists $(x,y) \in C_{n+1}$ such that $d_{n+1}(x,y)=(c,c)$, i.e.
\begin{equation} \label{p1}
\left\{
\begin{array}{l}
X_0(x) + X_1(y) = c \\ 
X_1(x) + X_0(y) = c.
\end{array}
\right.
\end{equation}
Observe 
$$c + d_{n}(y) = X_0(x) + X_1(y) + X_1(y) + X_0(y) = X_0(x+y).$$
This implies
\begin{equation} \label{p2}
c = X_0(x+y) + d_n(y).
\end{equation}
Note that (\ref{p1}) implies $d_n(x+y)=0$, which in turn gives $d_{n+1}(x+y,x+y)=0$. Again, by assumption, there is $(a,b)$ such that $d_{n+1}(a,b)=(x+y,x+y)$.
Using the same observation we have
$$x+y = X_0(a+b) +d_n(b).$$
Together with (\ref{p2}), this gives
\begin{eqnarray*}
c & = & X_0 \big( X_0(a+b) + d_n(b) \big) + d_n(y) \\
  & = & X_0^2(a+b) + X_0(d_n(b)) + d_n(y) \\
  & = & d_n(X_0(b) + y). 
\end{eqnarray*}
This contradicts the choice of $c$.
\end{proof}

Using the invariance of Floer cohomology under Hamiltonian diffeomorphism, the previous theorem immediately implies non-displaceability of the Lagrangians:

\begin{corollary}
Let $\phi$ be a Hamiltonian diffeomorphism of $X_P$. 
Then, under the hypotheses of Theorem 5.1,
$$\phi(R)\cap L_c \neq \emptyset.$$
\end{corollary}

We will now see how to find examples where Theorem \ref{nonzero} applies. 
We define the potential function
\begin{eqnarray*}
W : \textrm{int}(P) \times(\mathbb{F}^{\ast}_2)^n  & \rightarrow  & \unovf \\
      (c;x_1,\ldots,x_n) & \mapsto & \sum^{m}_{j=1}{x^{v^1_j}_{1} \cdots x^{v^n_j}_{n} T^{E_j}}.
\end{eqnarray*}
(Note that $E_j$ depends on $c$.)
Now fix $c \in \textrm{int}(P)$ and define $W_c(x) :=W(c,x)$. Also, put $y_i=\rho(l_i)$. Then
$$o(L_c,\rho)=W_c(y_1,\ldots,y_n),$$ 
$$Z_i=\sum^{m}_{j=1} v^i_j \rho^{v_j} T^{E_j}= (x_i\frac{\partial}{\partial x_i}W_c(x))|_{x=y}.$$
It follows that
\begin{corollary}
The Floer cohomology $HF(R,(L_c,\mathcal{L}_\rho);\unovf)$ is well-defined if and only if $W_c(y)=0$ and is non-zero if and only if  $\nabla W_c(y)=0$.
\end{corollary}

\begin{example}
As an example, let $P$ be the polytope determined by the following data:
\begin{eqnarray*}
v_1=(1,0,0), & v_2=(0,1,0), & v_3=(0,0,1), \\
v_4=(-1,-1,-1), & v_5=(1,1,1), \\
\lambda_1 = \lambda_2 = \lambda_3 = 0, & \lambda_4 = -4, & \lambda_5 = 2.
\end{eqnarray*}
As a complex manifold $X_P$ is the blow-up of $\C P^3$ at one point. Observe that for this choice of $\lambda$'s, $X_P$ is a monotone symplectic manifold. For $c=(1,1,1) \in \textrm{int}(P)$, the torus $L_c$ is a monotone Lagrangian submanifold. The potential function is
$$W_c(x_1,x_2,x_3)= \Big( x_1+x_2+x_3+(x_1 x_2 x_3)^{-1}+ x_1x_2x_3 \Big)T^{2\pi}.$$
The critical points of $W_c$ are given by the conditions
\begin{eqnarray*}
x_1=x_2=x_3, & x^3_1 + x_1 + x^{-3}_1=0,
\end{eqnarray*}
and $W_c$ is zero at these points.
Therefore we equip $L_c$ with the locally constant sheaf $\mathcal{L}_{\rho}$ determined by
$$\rho(l_i) = \xi$$
where $\xi \in \mathbb F_2^{\ast}$ satisfies
$$\xi^3 + \xi + \xi^{-3} = 0.$$
Then the Floer cohomology $HF(R, (L_c, \mathcal{L}_{\rho}))$ is well defined and does not vanish. 

We will now compute the cohomology explicitly. 
We pick the basis 
$$(000),(100),(010),(001),(011),(101),(110),(111)$$
for $CL(R,L_c)$. On this basis, the Floer differential is represented by the matrix

\begin{equation}
\left[
\begin{array}{cccccccc}
0 & \eta T & \eta T & \eta T & 0 & 0 & 0 & x T \\
\eta T & 0 & 0 & 0 & x T & \eta T & \eta T & 0 \\
\eta T & 0 & 0 & 0 & \eta T & x T & \eta T & 0 \\
\eta T & 0 & 0 & 0 & \eta T & \eta T & x T & 0 \\
0 & x T & \eta T & \eta T & 0 & 0 & 0 & \eta T \\
0 & \eta T & x T & \eta T & 0 & 0 & 0 & \eta T \\
0 & \eta T & \eta T & x T & 0 & 0 & 0 & \eta T \\
x T & 0 & 0 & 0 & \eta T & \eta T & \eta T & 0 \\
\end{array}
\right],
\end{equation}
where $\eta = \sqrt{\xi}$ and $x = \eta^3 + \eta^{-3}$. The choice of $\xi$ implies $x=\eta$. It follows that the matrix has rank 2. Therefore, $\dim(HF)=4$ and thus
$$ HF(R, (L_c, \mathcal{L}_{\rho}); \unovf)= (\unovf)^{4}. $$ 
\end{example}

Suppose now that $z$ is such that $\nabla W_c(z)=0$ but $W_c(z) \neq 0$. In this situation the Floer cohomology is not defined.
However, by following an idea of Miguel Abreu and Leonardo Macarini, we can still obtain a lower bound on the number of intersection points under Hamiltonian isotopy: 
Consider the polytope $Q=P \times P$. Then $X_Q = X_P \times X_P$, $R_Q = R_P \times R_P$ and $L_{c \times c} = L_{c} \times L_{c}$. From the definition of the potential we have
$$W_{c \times c}(x,y)=W_c(x) + W_c(y).$$
Since we are working in characteristic 2, $W_{c \times c}(z,z)=0$ so the Floer cohomology $$HF(R_Q,(L_{c \times c},\mathcal{L}_{z \times z}),\unovf)$$ is well-defined. Also $\nabla W_{c \times c}(z,z)=0$, which implies that the cohomology is non-zero. Using Hamiltonian invariance, we obtain

\begin{theorem}
Let $\phi$ be a Hamiltonian diffeomorphism of $X_P$ such that $\phi(R_P)$ and $L_c$ intersect transversely. Then
$$ \sharp \big( \phi(R_P) \cap L_c \big) \geq \sqrt{\textrm{rank} (HF(R_Q,(L_{c \times c},\mathcal{L}_{z \times z}),\unovf))} > 0 $$
\end{theorem}
\begin{proof}
By Hamiltonian invariance of Floer cohomology, we have
$$\textrm{rank} (HF(R_Q,(L_{c \times c},\mathcal{L}_{z \times z}),\unovf)) \leq \sharp \big( \psi(R_Q) \cap L_{c \times c} \big) $$
for $\psi$ any Hamiltonian diffeomorphism of $X_Q$. 
If we take $\psi = \phi \times \phi$, we get
\begin{eqnarray*}
\textrm{rank} (HF(R_Q,(L_{c \times c},\mathcal{L}_{z \times z}),\unovf))& \leq & \sharp \big( (\phi \times \phi)(R_P \times R_P) \cap (L_c \times L_c) \big) \\
  &= & \big( \sharp  ( \phi(R_P) \cap L_c) \big)^2. \\
\end{eqnarray*}
\end{proof}

We will now examine the case when the symplectic manifold is monotone, i.e. $[\omega_P]=\lambda c_1(X_P)$.
We rescale the symplectic form so that $\lambda=1$. In this case, $P$ is an integral polytope with the origin {0} in its interior and the structure constants $\lambda_i=-1$ (see \cite{mcduff} for a proof of this). 
Then $L_0$ is the unique monotone Lagrangian submanifold and $E_j$ is independent of $j$. 
\begin{lemma}
$W_0$ has at least one critical point.
\end{lemma}
\begin{proof}
$W_0$ has the form
$$W_0=\sum^{m}_{j=1}{x^{v^1_j}_{1} \cdots x^{v^n_j}_{n}\ T^E},$$
for some constant E.
We write $W_0=T^Ew_0$. Then, we can think of $w_0$ as an element of the ring $\mathbb{F}_2[x^{\pm1}_1, \ldots,x^{\pm1}_n ]$. The existence of a critical point of $W_0$ is equivalent to the existence of a critical point of $w_0$. Since $\mathbb{F}_2$ is algebraically closed, the non-existence of a crtitical point of $w_0$ is equivalent to the vanishing of the Jacobian ring 
$$Jac(w_0)=\frac{\mathbb{F}_2[x^{\pm1}_1, \ldots,x^{\pm1}_n ]}{\langle \partial_i w_0, i=0 \ldots n \rangle}.$$
In turn, this is equivalent to the existence of $f_i \in \mathbb{F}_2[x^{\pm1}_1, \ldots,x^{\pm1}_n ]$ such that
$$\sum^{n}_{i=1}f_i \partial_i w_0 =1,$$
which implies
$$\sum^{n}_{i=1}f_i \partial_i W_0 =T^E \ \in \unovf[x^{\pm1}_1, \ldots,x^{\pm1}_n ].$$
Since $T^E$ is invertible in $\unovf$, we obtain $Jac(W_0) \equiv0$. But, in \cite[Proposition 6.8]{toricfooo} it is proved that $Jac(W_0)$ is isomorphic to the Batyrev quantum cohomology $QH^{\omega}(X_P, \unovf)$, which is non-trivial.

So we conclude that $w_0$ has a critical point and therefore so does $W_0$.
\end{proof}

Combining this lemma with the previous theorem we obtain
\begin{corollary}
If $X_P$ is monotone, then the Lagrangians $R_P$ and $L_0$ are non-displaceable.
\end{corollary}

Let us now consider the case $X_P= \mathbb{C}P^k$, where $R_P= \mathbb{R}P^k$ and $L_0$ is the Clifford torus $T^k$.
If we don't include sheaves, i.e. $\rho \equiv 1$, we get 
$$W_0(1)=\sum^{m}_{j=1} T^E = mT^E=(k+1)T^E .$$ 
So the Floer cohomology is defined when $k$ is odd. This was the situation the first author studied in \cite{alston}, where it was proved that
\begin{equation} \label{odd}
HF(\mathbb{R}P^{2n-1},T^{2n-1}; \unov)= (\unov)^{2^n}.
\end{equation}
This proves that $\mathbb RP^{2n-1}$ and $T^{2n-1}$ always intersect in at least $2^n$ points under Hamiltonian deformation (if the intersection is transverse).

When $k$ is even we can use the theorem above. First we need to compute $HF(\mathbb{R}P^{2n} \times \mathbb{R}P^{2n}, T^{2n} \times T^{2n}; \unov)$. We will actually prove a slightly more general result:
\begin{proposition}
$$HF(\mathbb{R}P^{2k} \times \mathbb{R}P^{2j},T^{2k} \times T^{2j}) = (\unov)^{2^{k+j}}.$$
\end{proposition}
\begin{proof}
Let $P$ be a polytope in $\R^{2k+2j}$ defined by
\begin{displaymath}
\begin{array}{l}
v_i=e_i \ \textrm{for} \ i \in \{1,\ldots,2k+2j\} \\
v_{2k+2j+1}=-e_1-\ldots-e_{2k} \\
v_{2k+2j+2}=-e_{2k+1}\ldots,-e_{2k+2j} \\
\lambda_i=0 \ \textrm{for} \ i \in \{1,\ldots,2k+2j\} \\
\lambda_{2k+1}=-2k-1 \\
\lambda_{2k+2j+2}=-2j-1.
\end{array}
\end{displaymath}
The toric manifold associated to this polytope is $\C P^{2k} \times \C P^{2j}$ with the monotone symplectic form. If we take $c=(1,\ldots,1)$ then $L_c$ is the monotone Lagrangian $T^{2k} \times T^{2j}$.

By Theorem \ref{main}, we just need to compute the cohomology of the complex (\ref{formula}). We will denote the complex by $C(k,j)$ and write the differential $d_{k,j} = A + \eta$ where $A= \sum^{2k+2j+1}_{i=1} f_i$ and $\eta=f_{2k+2j+2}$. Note that in this description of the differential we are ignoring the energy terms. We can do this because in this situation, $E_i=T^{2\pi}$ for all $i$. This way the differential above differs from the one in (\ref{formula}) by a factor of $T^{2\pi}$ which is invertible in $\unov$.

We will use induction. For $k=j=1$ we can explicitly write down a matrix representing $d_{1,1}$. Then one gets that the rank of this matrix is 6 and so the cohomology has dimension 4 as we claimed. 

Note that there is an obvious symmetry in $i$ and $j$. So for the induction step, it's enough to assume the claim for $(k,j)$ and prove the case $(k,j+1)$.
The rest of the proof is very similar to the one in \cite{alston}.

We will denote $J=j+1$.
Let $\pi: C(k,J) \to C(k,j)$ be the map that forgets the last two coordinates, i.e.
$$\pi : (\epsilon_1,\ldots,\epsilon_{2k+2j},\epsilon_{2k+2j+1},\epsilon_{2k+2J}) \mapsto (\epsilon_1,\ldots,\epsilon_{2k+2j}).$$ 
This map is surjective and is a chain map: for $x \in C(k,j)$ and $a,b \in \Z_2$, we calculate
\begin{eqnarray*}
&& \pi \circ d_{k,J}(x,a,b) = \\
&& =  \pi \big( (A_{k,j}x,a,b)+(x,a+1,b)+(x,a,b+1)+(\eta_{k,j} x,a+1,b+1)\big)  \\
&& =  A_{k,j}x + x + x + \eta_{k,j} x = d_{k,j} x \\
&& = d_{k,j} \circ \pi(x,a,b).
\end{eqnarray*}
\textbf{Claim :} Let $x \in Ker(d_{k,J})$. Then $x$ can be uniquely written as
$$x=(u,0,0)+(u,1,1)+(v,1,0)+(v,0,1)+(w,0,0)+(w,1,0)+(0,0,t)$$
with $u,v \in C(k,j)$ and $w,t \in Ker(d_{k,j})$ satisfying
\begin{equation}\label{claim}
\left\{
\begin{array}{l}
d_{k,j}v=w+t+\eta w,\\
d_{k,j}u=w+\eta t+\eta w.
\end{array}
\right.
\end{equation} 
This claim can be proved in the same way as Lemma 4.3 in \cite{alston}.

Now, consider the map
$$\alpha:C(k,j) \oplus C(k,j) \oplus Ker(d_{k,j}) \oplus Ker(d_{k,j}) \to \Ima(d_{k,j}) \oplus Ker(d_{k,j})$$
$$(u,v,w,t)\mapsto(d_{k,j}\eta u+d_{k,j} v,d_{k,j}\eta u +w+t+\eta w).$$
This map is clearly onto. Also, using the fact that $\eta^2=id$, we can check that 
$$ Ker(\alpha)=\big\{(u,v,w,t) \ \textrm{satisfying} \ (\ref{claim}) \big\}.$$
This shows
$$\dim(\Ker(d_{k,J})=\dim(\Ker(\alpha))$$
$$=2\dim(C(k,j))+2\dim(\Ker(d_{k,j}))-\dim(\Ima(d_{k,j}))-\dim(\Ker(d_{k,j}))$$
$$=2\dim(C(k,j))+\dim(H^{\ast}(d_{k,j})).$$
Therefore
\begin{eqnarray*}
\dim(H^{\ast}(d_{k,J}))&=& 2\dim(\Ker(d_{k,J}))-\dim(C(k,J)) \\
&=&  4\dim(C(k,j))+2\dim{H^{\ast}(d_{k,j})}-4\dim(C(k,j)) \\
&=&  2\dim{H^{\ast}(d_{k,j})}
\end{eqnarray*}
This completes the inductive step.
\end{proof}

Now we can apply the previous theorem to this case. Combining with (\ref{odd}), we get
\begin{corollary} 
$$\sharp \big( \psi ( \mathbb{R}P^{k} ) \cap T^{k} \big) \geq 2^{ \big\lceil \frac{k}{2} \big\rceil }. $$
\end{corollary}

We believe that this bound is optimal. Although we do not have a proof of this fact, we have checked it in dimension 2. 
Through experimentation with the help of a computer, we have found a matrix $M \in U(3)$ that realizes the lower bound
$$\sharp \big( M ( T^2 ) \cap \mathbb{R}P^{2} \big) = 2.$$
Note that $U(n+1)$ acts on $\C P^n$ by Hamiltonian diffeomorphisms. This shows that in dimension 2 the bound we obtained is sharp. We expect the same to hold in higher dimensions.

\section{Lagrangian Floer Theory}

Recall that our approach to the Floer cohomology of $(R,(L_c,\mathcal L_\rho))$ is to use the theory developed in \cite{fooo} (and its extension to the $\Z$-coefficient case developed in \cite{fooo8}) and then show that this reduces to simply counting Maslov index 1 holomorphic strips (for the standard toric complex structure).
In this section we will review the Floer theory from \cite{fooo} and \cite{fooo8}. 
The explanation of how the obstructions $o(L_c),o(R)$ are to be viewed in this context and the proof that the Floer differential simply counts strips will be given in the next section.

Let $ (M,\omega) $ be a symplectic manifold and $J$ an almost complex structure compatible with $\omega$. Let 
$\mathcal{M}^{J}_{k+1}(L;\beta)$ 
be the moduli space of stable maps of genus zero with boundary on $L$ and in a fixed homotopy class $\beta \in \pi_2(M,L)$. 
$\mathcal{M}^{J}_{k+1}(L;\beta)$ is the compactification of
$$ \Big\{ \big(u,(z_0, \ldots, z_k)\big) \Big| \ u:(D^2,\partial D^2)\rightarrow(M,L), \ \bar{\partial}_J u=0, \ z_i \in \partial D^2, \  [u]=\beta \Big\} \Big/ PSL(2,\R). $$
This moduli space has natural evaluation maps $ev_i:\mathcal{M}^{J}_{k+1}(L;\beta) \to L $ given by
$$ev_i \big( u,(z_0,\ldots,z_k) \big)=u(z_i) \in L, \ i=0,\ldots,k.$$
Let $C(L)$ be a suitably chosen subcomplex of the singular chain complex of $L$ (see Section 7.2 of \cite{fooo} for the detailed construction of $C(L)$) with coefficients in $\Z_2$. Given $k$ chains $P_1,\ldots,P_k \in C(L)$ consider the fiber product
\begin{equation}\label{fiber}
\mathcal M^{J}_{k+1}(\beta;L;P_1,\ldots,P_k):=\mathcal{M}^{J}_{k+1}(\beta;L) \times _{(ev_1,\ldots,ev_k)}(P_1 \times \ldots \times P_k).
\end{equation}
Then
\begin{theorem}[\cite{fooo8} Theorem 34.11]\label{thm:main_kuranishi_space_thm}
The moduli space (\ref{fiber}) has a Kuranishi structure. If $J$ is spherically positive (the first Chern class is positive on every non-constant J-holomorphic sphere) then the Kuranishi structure admits a family $\mathfrak s^\epsilon$ of single-valued piece-wise smooth sections of the obstruction bundle and a decomposition 
\begin{eqnarray*}
&\mathcal M^{J}_{k+1}(\beta;L;P_1,\ldots,P_k)^{\mathfrak s^\epsilon}=&\\
&\mathcal M^{J}_{k+1}(\beta;L;P_1,\ldots,P_k)^{\mathfrak s^\epsilon}_{free}\cup \mathcal M^{J}_{k+1}(\beta;L;P_1,\ldots,P_k)^{\mathfrak s^\epsilon}_{fix}
\end{eqnarray*}
of the perturbed moduli space such that
\begin{enumerate}
\item $\mathcal M^{J}_{k+1}(\beta;L;P_1,\ldots,P_k)^{\mathfrak s^\epsilon}_{free}$ is a PL manifold,\\
\item $\mathcal M^{J}_{k+1}(\beta;L;P_1,\ldots,P_k)^{\mathfrak s^\epsilon}$ has a triangulation,\\
\item $\mathcal M^{J}_{k+1}(\beta;L;P_1,\ldots,P_k)^{\mathfrak s^\epsilon}_{fix}$ is contained in a subcomplex of dimension less than or equal to
$$\dim \mathcal M^{J}_{k+1}(\beta;L;P_1,\ldots,P_k)^{\mathfrak s^\epsilon}_{free}-2,$$\\
\item and $\lim_{\epsilon\to0}\mathfrak s^{\epsilon}=\mathfrak s,$ where $\mathfrak s=\bar\partial$ is the unperturbed section.
\end{enumerate}
\end{theorem}

\begin{rmk}
This theorem is the basic result needed to define Floer cohomology and is the analog of Proposition 7.2.35 in\cite{fooo} where $\mathbb Q$-coefficients are used.
The ``free'' subscript on the moduli spaces refers to points with trivial isotropy and the ``fix'' subscript refers to points with non-trivial isotropy.
The fact that the set of points with non-trivial isotropy has codimension at least two is the essential feature that allows $\Z$-coefficients to be used; this fact in turn hinges on the spherical positivity of the complex structure $J$.
Since these issues are at the heart of what allows us to use $\Z_2$-coefficients in this paper we will take some time to explain them in detail.

Let us begin by looking at a similar (but easier to understand) situation that arises with closed rational Gromov-Witten invariants, following the approach taken in \cite{ms}.
Namely, let $(M,\omega,J)$ be a spherically postive symplectic manifold, and let $\mathcal M_k^J(A)$ be the moduli space of $J$-holomorphic spehers of homotopy class $A\in\pi_2(M)$ with $k$ marked points.
If the spheres in $\mathcal M_k^J(A)$ are regular, then the smooth points (i.e. interior and non-orbifold points) of $\mathcal M_k^J(A)$ will form a manifold and Gromov-Witten invariants can be defined by doing intersection theory with this moduli space.
In order for the intersection theory to be well-behaved $\mathcal M_k^J(A)$ must be a pseudo-cycle, that is the non-smooth points must have codimension two or more.
To be more precise, the non-smooth points must have actual codimension two not just virtual codimension two because no perturbations are being done.
(Boundary points always have virtual codimension at least two.)
Orbifold points come from multiply covered spheres, and non-interior points come from maps with more than one component.
Under the assumption that $c_1>0$ for all $J$-holomorphic spheres and the spheres are regular, it is easy to see using standard dimension counting arguments that these points are of actual (real) codimension two or more.

The above construction shows that the Gromov-Witten invariants are well-defined as integers in the spherically positive case.
If the manifold is not spherically positive, then this argument does not work, and a more elaborate construction such as that in \cite{fo-acgwi} needs to be used.
In particular, the orbifold points cannot be ignored, perturbations need to be done to ensure that the boundary has the correct codimension, and in general the invariants are rational numbers instead of integers.

Now consider the situation of Theorem \ref{thm:main_kuranishi_space_thm}.
We will first explain the properties that the single-valued section constructed in \cite{fooo8} has, and then explain why spherical positivity implies that the fixed point part has actual codimension at least two (not just virtual codimension two).

The first thing to be said about the single-valued section $\mathfrak s^\epsilon$ is that it is not in general transverse to the zero-section.
Instead, locally the situation is as follows: let $U\subset \R^n$ be an open set, $\Gamma$ a finite group acting on $U$, $E$ an equivariant vector bundle over $U$ (i.e. there is also a $\Gamma$ action on $E$), and $s:U\to E$ is a section.
Let $E^\Gamma\subset E$ be the subbundle (maybe subsheaf is a more correct term) fixed by the $\Gamma$ action, that is for $x\in U$ the fiber is
\begin{displaymath}
E^\Gamma_x=\set{v\in E_x}{\gamma\cdot v=v\, \forall \gamma\in \Gamma_x}
\end{displaymath}
and let $E^\perp$ be a complement to $E^\Gamma$.
A single-valued perturbation $s_\epsilon$ of $s$ can be constructed such that if $s_\epsilon(x)=0$ then $s_\epsilon$ when restricted to the submanifold $U(\Gamma_x)=\set{y\in U}{\Gamma_x\sim\Gamma_y}$ is transverse to the zero section of $E^{\Gamma}|U(\Gamma_x)$.
Setting the $E^\perp$ component of $s_\epsilon|U(\Gamma_x)$ to $0$ then makes $s_\epsilon|U(\Gamma_x)$ a $\Gamma$-equivariant section.
It follows that $s^{-1}_\epsilon(0)\cap U(\Gamma_x)$ is a smooth manifold of dimension $\dim U(\Gamma_x)-\dim E^{\Gamma}_x$, and hence has a smooth triangulation of the same dimension that descends to the quotient $(s^{-1}_\epsilon(0)\cap U(\Gamma_x))/\Gamma$.
The triangulations of $(s^{-1}_\epsilon(0)\cap U(\Gamma_x))/\Gamma$ for different $\Gamma_x$'s can be pieced together to give a triangulation of $s^{-1}_\epsilon(0)/\Gamma$.
If the action of $\Gamma$ is effective, then the top dimensional stratum will have the correct dimension, namely $\dim U-\textrm{rank } E$.
A stratum coming from $U(\Gamma_x)$ with non-trivial isotropy, which has virtual codimension equal to the codimension of $U(\Gamma_x)$ in $U$, will have actual codimension equal to this plus $\dim E^\perp_x$ because the actual dimension of $(s^{-1}_\epsilon(0)\cap U(\Gamma_x))/\Gamma$ is $\dim U(\Gamma_x)-\dim E^{\Gamma}_x$.
(For simplicity we have tacitly assumed that $U(\Gamma_x)$ has only one component; if it has more than one component the same construction works by considering the different components seperately.
In \cite{fooo8} they use the notation $U(\Gamma_x,i)$ to refer to the different components of $U(\Gamma_x)$.)

A general Kuranishi space $X$ with single-valued section $\mathfrak s$ locally looks like the situation described in the previous paragraph.
The local single-valued perturbations and triangulations can be glued together to get a single-valued perturbation $\mathfrak s^\epsilon$ of $\mathfrak s$ along with a triangulation of its zero-set.
The top-dimensional stratum of triangulation of the zero-set will have the correct dimension, i.e. its actual dimension will equal its virtual dimension.
But the simplices in the triangulation of the lower dimensional stratum (stratum coming from non-trivial isotropy groups) will have actual dimension that differs from the expected dimension (i.e. virtual dimension) in the same way as described above.
This is the content of Proposition 35.52 in \cite{fooo8}.

Applying the proposition to the Kuranishi space $\mathcal M^{J}_{k+1}(\beta;L;P_1,\ldots,P_k)$ proves most of Theorem \ref{thm:main_kuranishi_space_thm}:
$\mathcal M^{J}_{k+1}(\beta;L;P_1,\ldots,P_k)^{\mathfrak s^\epsilon}$ is the zero-set of the perturbed section $\mathfrak s^\epsilon$, it has a triangulation such that the smooth part is a manifold of the correct dimension, and the decomposition
\begin{eqnarray*}
&\mathcal M^{J}_{k+1}(\beta;L;P_1,\ldots,P_k)^{\mathfrak s^\epsilon}=&\\
&\mathcal M^{J}_{k+1}(\beta;L;P_1,\ldots,P_k)^{\mathfrak s^\epsilon}_{free}\cup \mathcal M^{J}_{k+1}(\beta;L;P_1,\ldots,P_k)^{\mathfrak s^\epsilon}_{fix}
\end{eqnarray*}
is the decomposition into the parts with trivial and non-trivial isotropy, which is respected by the triangulation.
The final point to address is item 3 which states that the (actual) codimension of the triangulation of the fixed part is at least two.
The proof of this item uses spherical positivity, which up to this point we have not yet used.

To address item 3 we need to explain why $\mathcal M^{J}_{k+1}(\beta;L;P_1,\ldots,P_k)_{fix}^{\mathfrak s^\epsilon}$ has actual codimension at least two.
The inclusion of $P_1,\ldots,P_k$ makes no essential difference in this argument so it suffices to explain why $\mathcal M^{J}_{k+1}(\beta;L)_{fix}^{\mathfrak s^\epsilon}$ has actual codimension at least two in $\mathcal M^{J}_{k+1}(\beta;L)^{\mathfrak s^\epsilon}$.
Recall that a curve is a nodal Riemann surface (with boundary) with maybe multiple components; a map is a holomorphic map from a curve into $M$.
Since the curves underlying the elements of $\mathcal M^{J}_{k+1}(\beta;L)$ have at least one disc component with at least one special boundary point, and such disc components have no non-trivial automorphisms, the elements of $\mathcal M^{J}_{k+1}(\beta;L)$ with non-trivial isotropy have to have at least one sphere component.
(Thus the virtual codimension of the fixed part is at least two, but again we need more than this.)
If $x=[u]\in\mathcal M^{J}_{k+1}(\beta;L)_{fix}^{\mathfrak s^\epsilon}$ is a point in the interior of a simplex contained in the triangulation of the fixed part, then, using the notation from before, the dimension of the simplex is $\dim U(\Gamma_x)-\dim E^{\Gamma}_x$.
Subtracting the number of deformation parameters of the underlying curve that keep the same combinatorial type of the curve (i.e. deforming the curve within its stratum in the moduli space of all curves) from this dimension gives the equivariant index of the linearized $\bar\partial$ complex at u, see Lemma 35.72 in \cite{fooo8}.

By Proposition 35.63 and Lemma 35.74 in \cite{fooo8}, this equivariant index is equal to the virtual dimension of the so-called reduced marked model.
The underlying curve of the reduced marked model is the underlying curve of the map $[u]$ modded out by $\Gamma$, with additional marked points added to make the curve a stable curve.
The map $[u]$ descends to the reduced marked stable curve, and this is the reduced marked model.

The crux of the matter is then this:
Spherical positivity implies that the maps on the sphere components of the reduced marked model with extra marked points added have Chern number strictly less than the Chern number of the corresponding maps in the original $u$ (that is, the maps in $u$ are multiple covers of the maps in the reduced marked model).
This, in addition to the fact that the reduced marked model has at least one sphere component, is enough to show that the virtual dimension of the moduli space of reduced marked models is at least two less than the virtual dimension of $\mathcal M^{J}_{k+1}(\beta;L)$ (Lemma 35.75 in \cite{fooo8}).
Putting everything together shows that the actual codimension of $\mathcal M^{J}_{k+1}(\beta;L)_{fix}^{\mathfrak s^\epsilon}$ in $\mathcal M^{J}_{k+1}(\beta;L)^{\mathfrak s^\epsilon}$ is at least two.

Now, with Theorem \ref{thm:main_kuranishi_space_thm} in hand, and similar statements for other moduli spaces, the machinery developed in \cite{fooo} goes through with $\Z$-coefficients.
Indeed, the perturbation $\mathfrak s^\epsilon$ is equivariant and single-valued so the triangulation of its zero set can be described as a singular chain with $\Z$-coefficients.
The same holds for other moduli spaces and thus the Floer theory is defined over $\Z$.
To be more explicit, let us highlight the reason that the fixed part of the moduli space needs to have codimension at least two:
The single valued perturbation is not transverse to the zero section, but the machinery developed in \cite{fooo} always assumes that the perturbations are transverse.
The fixed part having actual codimension at least two implies that the part where the section is not transverse also has actual codimension at least two.
Since the $A_\infty$-relations underlying Floer theory are a consequence of the properties of the codimension one boundary of the moduli space, non-transversality on codimension two does not affect the validity of the arguments.
(Compare for example to the notion of a pseudo-cycle as in \cite{ms}, mentioned above in the context of closed Gromov-Witten theory.)
\end{rmk}

We define
  \begin{eqnarray}
\nonumber
  m^{J}_{0,\beta}(1)&=&\left\{
  \begin{array}{ll}
  ev_{0*}[\mathcal M^{J}_1(\beta;L)^{\mathfrak s}] & \beta\neq0\\
  0 & \beta=0
  \end{array}
  \right.,\\
  m^{J}_{1,\beta}(P,f)&=&\left\{
  \begin{array}{ll}
  ev_{0*}[\mathcal M^{J}_2(\beta;L)\times_{ev_1}P^{\mathfrak s}] & \beta\neq0\\
  \partial P & \beta=0
  \end{array}
  \right.,\\
\nonumber
  m^{J}_{k,\beta}(P_1,\ldots,P_k)&=&ev_{0*}[\mathcal M^{J}_{k+1}(\beta;L)\times_{ev_1\times\cdots\times ev_k}(P_1\times\cdots\times P_k)^{\mathfrak s}] \textrm{ for $k\geq2$.}
  \end{eqnarray}
Now consider $\unovof=\Big\{ \sum^{\infty}_{i=1}{a_iT^{\lambda_i}} \in \unovf \Big| \ \lambda_i \geq 0 \Big\}$ and let $C(L,\unovof)=C(L) \hat{\otimes}_{\Z_2} \unovo$ be the completion of the tensor product with respect to the natural filtration on $\unovo$. 

\begin{rmk}
$C(L,\unovof)$ inherits a grading from $C(L)$ given by the codimension of the chains. However the operations we defined do not respect this grading, for example $m_1$ is not homogeneous of degree 1. 
We could fix this by adding to the Novikov ring a parameter $e$ of degree 2 as is done in \cite{fooo}. 
We do not do this since we prefer to work with the simpler Novikov ring; as a consequence the Floer cohomology does not have a grading.

We would also like to point out that $C(L)$ has chains of dimension greater than the dimension of $L$.
Indeed, this is true for the full singular chain complex (of course cycles with dimension greater than $\dim L$ will be $0$ in homology), and it is also true for the suitably chosen $C(L)$ because $C(L)$ needs to contain the images of all the evaluation maps.
\end {rmk}

Now let $\mathcal{L}_{\rho}$ be a locally constant sheaf on $L$.
We then define 
$$m^{\rho,J}_k:C(L,\unovof)^{\otimes k}\to C(L,\unovof)$$ by
\begin{equation}
m^{\rho,J}_k=\sum_{\beta \in \pi_2(M,L)} \rho(\partial \beta)m^{J}_{k,\beta} T^{\omega(\beta)} \in C(L,\unovof).
\end{equation}
With these definitions we get to the main result of \cite{fooo}:
\begin{theorem}
$\big( C(L,\unovof), m^{\rho,J}_k, k\geq0 \big)$ is a filtered $A_{\infty}$-algebra, i.e.
$$\sum_{0\leq k\leq n, i} m^{\rho,J}_{n-k+1}(x_1,\ldots,m^{\rho,J}_k(x_i, \ldots, x_{i+k-1}),\ldots,x_n)=0.$$
\end{theorem}

When working over $\mathbb Q$ the homotopy type  of this $A_{\infty}$-algebra (see \cite[Chapter 4]{fooo} for the definition of homotopy) is independent of the almost complex structure. More precisely, it is proved in Section 4.6 of \cite{fooo} that a path $J_t$ ($0\leq t\leq 1$) induces a homotopy from $\big( C(L,\unovof), m^{\rho,J_0}_k \big)$ to $\big( C(L,\unovof), m^{\rho,J_1}_k,\big)$. 
The independence then follows from the fact that the space of compatible almost complex structures is connected (in fact contractible).  
When working with integer coefficients we have to restrict ourselves to spherically positive complex structures, but the space of these is not connected. So the homotopy type of $\big( C(L,\unovof), m^{\rho,J}_k \big)$ depends on the connected component of $J$ in the space of spherically positive complex structures.

The next concept to consider is that of a weak bounding cochain.
This will then allow us to define Floer cohomology.
\begin{definition}
Let $(C,m^{\rho,J}_k)$ be the filtered $A_{\infty}$-algebra from above. An element $b \in T^{\lambda}C$ ($\lambda >0$) is called a weak bounding cochain if it satisfies the Maurer-Cartan equation
$$\sum^{\infty}_{k=0} m^{\rho,J}_k(b,\ldots,b)= \mathfrak{P}(L,\rho,b,J)[L].$$
Here $[L]$ is the fundamental class of the Lagrangian submanifold and $\mathfrak{P}(L,\rho,b,J)$ is an element in $\unovof$.

If in addition $\mathfrak{P}(L,\rho,b,J)=0$ then $b$ is called a bounding cochain. 
\end{definition}

\begin{definition}
$(L,\mathcal L_\rho)$ is called (weakly) unobstructed if $(C(L,\unovof), m^{\rho,J}_k)$ has a (weak) bounding cochain. 
If $b$ is a weak bounding cochain let $\delta_{b,J} : C(L,\unovof) \rightarrow C(L,\unovof)$ be
$$\delta_{b,J}(P)=\sum_{i,j} m^{\rho,J}_{i+j+1}(\underbrace{b,\ldots,b}_{i},P,\underbrace{b,\ldots,b}_{j}).$$
The Maurer-Cartan equation implies that $\delta_{b,J}$ is a differential. The Floer cohomology is then defined to be
$$HF((L,\mathcal{L}_{\rho}),b,J;\unovof)=\frac{Ker \ \delta_{b,J}}{Im \ \delta_{b,J}}.$$
\end{definition}

\begin{rmk}
1. When using the graded Novikov ring as in \cite{fooo} weak bounding cochains are homogeneous of degree one, i.e. $b=\sum_{i\geq 0}{b_i e^{\mu_i /2} T^{\lambda_i}}$ where $b_i$ is a chain of codimension $1-\mu_i$.
Recall, however, that we are not using the graded Novikov ring in this paper.

2. In \cite{toricfooo}, instead of deforming the $A_ \infty $-operations $m_k$  by using the local system $\mathcal L_\rho$ the authors include $\rho$ in the weak bounding cochain. 
In this approach, $C(L)$ is replaced by the de Rham complex $\Omega(L)$ with complex coefficients. 
Then $\mathcal L_\rho$ is just a flat line bundle which can be represented by a closed one form $\rho$.  
The bounding cochains considered are then of the form $b'=\rho + b$ where $b \in T^{\lambda}\Omega(L)$.
With this approach, convergence of the Maurer-Cartan equation then becomes an issue.
It turns out that for toric fibers the sum does converge.
\end{rmk}

Now consider the case of a pair of Lagrangians $L,L'$ intersecting transversely equipped with locally constant sheaves $\mathcal{L}_{\rho}$ and $\mathcal{L'}_{\rho'}$ on $L$ and $L'$. Define 
$$C\big( (L,\mathcal{L}),(L',\mathcal{L'});\unovof \big) := \bigoplus_{p \in L \cap L'} \textrm{Hom}(\mathcal{L'}|_p,\mathcal{L}|_p) \otimes \unovof.$$ 
Observe that $\textrm{Hom}(\mathcal{L'}|_p,\mathcal{L}|_p)$ is isomorphic to $\mathbb{F}_2$ but not canonically so.
Now take a path of spherically positive almost complex structures $J_t$ ($0\leq t\leq 1$) (general Floer theory allows a one-parameter family of almost complex structures; for the applications in this paper we will just use a fixed complex structure) and consider the moduli space of \textit{stable} holomorphic strips. Take $p,q \in L \cap L'$,  and define the space

\begin{equation}\label{dbar}
\left\{
\big[ u,(z_1,\ldots,z_i),(z'_1,\ldots,z'_j)\big]
\Bigg|
\begin{array}{l}
z_m \in \R \times \{0\}, \ z'_n \in \R \times \{1\} \\
u:\mathbb R\times[0,1] \rightarrow M \\
\bar\partial_J u=\frac{\partial}{\partial s}u+J_t(u) \frac{\partial}{\partial t} u=0, \\
u(\cdot,0)\in L_0,\,u(\cdot,1)\in L_1, \\
u(-\infty,\cdot)=p, u(+\infty,\cdot)=q 
\end{array}
\right\}
\end{equation}
modulo automorphisms of the strip that identify all the data. Fix a homotopy class $B \in \pi_2(p,q)$ of trajectories and take the stable map compactification of this space. Denote it by $\mathcal{M}^{J_t}_{i,j}(p,q;B)$.
Now given $P_1,\ldots,P_i \in C(L;\unovo)$ and $P'_1,\ldots,P'_j \in C(L';\unovo)$ we take the fiber product
\begin{eqnarray*}
\lefteqn{\mathcal{M}^{J_t}_{i,j}(p,q;B;P_1,\ldots,P_i,P'_1,\ldots,P'_j) =}  \\
&& \mathcal{M}^{J_t}_{i,j}(p,q;B)\times _{(ev_1,\ldots,ev_i, ev'_1,\ldots,ev'_j)}(P_1 \times \ldots \times P_i, P'_1\times \ldots \times P'_j) 
\end{eqnarray*}
where $ev_m$ and $ev'_n$ are the evaluation maps at the marked points.

As in the case of one Lagrangian, it is proved in \cite[Chapter 7]{fooo} that this moduli space has a Kuranishi structure and, under the spherical positivity assumption, it is proved in \cite[Section 35]{fooo8} that it has a fundamental chain with coefficients in $\Z_2$. 
Now assume that $(C(L),m^{\rho,J_0}_k)$ and $(C(L'),m^{\rho',J_1}_k)$ have weak bounding cochains $b$ and $b'$ and define the operators
$$n^{J_t}_{i,j}: C\big( (L,\mathcal{L}),(L',\mathcal{L'});\unovof \big) \rightarrow C\big( (L,\mathcal{L}),(L',\mathcal{L'});\unovof \big)$$
by
\begin{eqnarray}\label{bimod}
n^{J_t}_{i,j}(\alpha_p) &=& \sum_{q,B} T^{\omega(B)} \sharp (\mathcal{M}^{J_t}_{i,j}(p,q;B;\underbrace{b,\ldots,b}_{i},\underbrace{b',\ldots,b'}_{j}))\alpha_q,
\end{eqnarray}
where the sum is over all $(q,B)$ such that 
$\textrm{vir dim } \mathcal{M}^{J_t}_{i,j}(p,q;B,b,\ldots,b,b',\ldots,b')=0$ and the symbol $\sharp$ stands for the number of points in this zero dimensional (compact) space.

Here, $\alpha_q \in \textrm{Hom}(\mathcal{L'}|_q,\mathcal{L}|_q)$ is defined as follows: $B \in \pi_2(p,q)$ determines (homotopy classes of) paths $l_0$ in $L$ and $l_1$ in $L'$ from $q$ to $p$. By taking constant sections of $\mathcal{L'}$ over $l_1$ we obtain a map in $\textrm{Hom}(\mathcal{L'}|_q,\mathcal{L'}|_p)$ that we denote by $par(l_1)$. Let $l_0^{-1}$ be the path $l_0$ with the opposite orientation and define $par(l_0^{-1})$ in a similar way. We finally put
$$\alpha_q=par(l_0^{-1})\circ \alpha_p \circ par(l_1).$$
\begin{rmk}
Note that this definition of $\alpha_q$ reduces to the one we gave in Section 3 where we took $\mathcal{L}'$ to be constant, since in that case $\mathcal{L}'|_p$ is canonically isomorphic to $\mathbb{F}_2$ and $par(l_0^{-1})$ is the identity.
\end{rmk}

Finally define the map 
\begin{equation}
\delta_{b,b',J_t}(\alpha_p)=\sum_{i,j \geq 0} n^{J_t}_{i,j}(\alpha_p).
\end{equation}

\begin{theorem}\cite[Section 3.7]{fooo}\label{obs}
Let $(L,\mathcal{L}_{\rho})$ and $(L',\mathcal{L'}_{\rho'})$ be as above with weak bounding cochains $b$ and $b'$. Then $$\delta_{b,b',J_t}^2 = (\mathfrak{P}(L,\rho,b,J_0)+\mathfrak{P}(L',\rho',b',J_1) )id.$$
\end{theorem}
\begin{definition}
If $\mathfrak{P}(L,\rho,b,J_0)=\mathfrak{P}(L',\rho',b',J_1)$ then the Floer cohomogy of the pair $$ \big( (L,\mathcal{L}_{\rho},b),(L',\mathcal{L'}_{\rho'},b') \big)$$ is defined to be
\begin{equation}
HF\big( (L,\mathcal{L}_{\rho},b),(L',\mathcal{L'}_{\rho'},b'), J_t ;\unovof\big) = \frac{\textrm{Ker} \ \delta_{b,b',J_t}}{\textrm{Im} \ \delta_{b,b',J_t}}.
\end{equation}
Observe that generally this cohomology depends not only on the Lagrangians and local systems but also on the weak bounding cochains.
\end{definition}

Floer cohomology with coefficients in the ring $\unovof$ is not invariant under Hamiltonian isotopy.
However, if the ring is changed to $\unovf$, it is invariant under Hamiltonian isotopy.
$\unovf$ is the field of fractions of $\unovof$, so it is a flat $\unovf$-module.
Thus, the Floer cohomology with $\unovf$ coefficients is simply
\begin{displaymath}
HF\big( (L,\mathcal{L}_{\rho},b),(L',\mathcal{L'}_{\rho'},b'), J_t ;\unovf\big) = HF\big( (L,\mathcal{L}_{\rho},b),(L',\mathcal{L'}_{\rho'},b'), J_t ;\unovof\big)\otimes_{\unovof}\unovf.
\end{displaymath}
The invariance under Hamiltonian isotopy is proved in \cite[Section 5.3]{fooo} for rational coefficients and in \cite[Theorem 34.3]{fooo8} for integral coefficients.
Here is the precise statement:
\begin{theorem}
Let $\psi_t$ be a Hamiltonian isotopy in $(M,\omega)$ and $J_t$ a path of spherically positive complex structures. Then $J^{\psi}_t=(\psi_{1-t})_{*}J_t$ is spherically positive.

Then $\psi_1$ induces a map of $A_{\infty}$-algebras $\psi_{*} :(C(L;\unovf), m^{\rho,J_0}_k) \to (C(\psi(L);\unovf), m^{\psi_{\ast}\rho,J^{\psi}_0}_k )$. Moreover, $\psi_{\ast}(b)$ is a weak bounding cochain and there is an isomorphism
$$ HF \big( (\psi_1(L),\psi_*(\mathcal{L}_{\rho}),\psi_{\ast}(b)),(L',\mathcal{L'}_{\rho'},b'), J^{\psi}_t ;\unovf\big) \cong HF\big( (L,\mathcal{L}_{\rho},b),(L',\mathcal{L'}_{\rho'},b'), J_t ;\unovf\big) .$$
\end{theorem}

Since the complex $C\big( (L,\mathcal{L}),(L',\mathcal{L'}), J_t;\unovf \big)$ is generated by intersections points, this theorem implies

\begin{corollary}
Let $\psi_t$ be a Hamiltonian isotopy such that $\psi_1(L)$ and $L'$ intersect transversely. Then
\begin{eqnarray*}
 \sharp \big( \psi_1(L) \cap L' \big) & \geq & \textrm{rank}  \ HF \big( (\psi_1(L),\psi_*(\mathcal{L}_{\rho}),\psi_{\ast}(b)),(L',\mathcal{L'}_{\rho'},b'), J^{\psi}_t ;\unovf\big) \\
 & = & \textrm{rank}  \ HF\big( (L,\mathcal{L}_{\rho},b),(L',\mathcal{L'}_{\rho'},b'), J_t ;\unovf\big) .
\end{eqnarray*}
\end{corollary}

We finish this section with a proposition that will be used in the next section.
\begin{proposition}\label{diff}
Let $ \big( (L,\mathcal{L}_{\rho}),(L',\mathcal{L'}_{\rho'}) \big)$ be a pair of weakly unobstructed Lagrangians as above. Let $b=\sum_{i\geq 0}{b_i T^{\lambda_i}}$ and $b'=\sum_{i\geq 0}{b'_i T^{\lambda'_i}}$ be weak bounding cochains of $L$ and $L'$, such that $\dim{b_i}, \ \dim{b'_i} \geq n+1, \ \forall i$ (here $n=\dim L$).
If $ \mathcal{M}^{J_t}_{i,j}(p,q;B)= \emptyset $ for all $B$ with Maslov index $\mu(B)\leq 0$ then
$$\delta_{b,b',J_t}(\alpha_p)=n^{J_t}_{0,0}(\alpha_p).$$
\end{proposition}
\begin{proof}
This follows from an easy dimension counting argument.
Namely, suppose $$\mathcal{M}^{J_t}_{i,j}(p,q;B,b_{m_1},\ldots,b_{m_i},b'_{n_1},\ldots,b'_{n_j})$$ contributes to $\delta_{b,b',J_t}$.
Then
\begin{eqnarray}
\lefteqn{0=\textrm{vir dim }\mathcal{M}^{J_t}_{i,j}(p,q;B,b_{m_1},\ldots,b_{m_i},b'_{n_1},\ldots,b'_{n_j})} \nonumber \\
&& = \mu(B)+i+j+ \sum{\dim{b}}+ \sum{\dim{b'}}-(i+j)n-1 \nonumber \\
&& \geq \mu(B)+i+j+ i(n+1)+ j(n+1)-(i+j)n-1  \\
&& \geq \mu(B)+2(i+j)-1.\nonumber
\end{eqnarray}
The  first equality is the standard dimension formula, proved in \cite[Proposition 3.7.36]{fooo}. The first inequality follows from the assumptions on $b$ and $b'$. 
Therefore
\begin{equation}
\mu(B) \leq  1-2(i+j).
\end{equation}
The assumption on the Maslov index of holomorphic strips then implies $i+j=0$.
This proves that $\delta_{b,b',J_t}(\alpha_p)=n^{J_t}_{0,0}(\alpha_p)$.
\end{proof}

\section{Equality of Floer chain complexes}
In this section we will show that the Lagrangians $R_P$ and $(L_c,\mathcal L_\rho)$ admit weak bounding cochains $b_P$ and $b_c$. 
Finally, we will then show that for these bounding cochains the Floer differential $\delta_{b,b'}$ defined in the previous section reduces to the one we described in (\ref{main}).
Since we are assuming $X_P$ is Fano the toric complex structure is spherically positive and we can use $\Z_2$-coefficients. 
Throughout this section the toric complex structure is the only complex structure we will use so we will drop it from the notation.

Let $\tau : X_P\rightarrow X_P$ be the complex conjugation map.
Let $u:(D^2,\partial D^2) \rightarrow (M,R)$ be a holomorphic disc with boundary in $R_P$ and take $\bar{u}=\tau \circ u$ defined on the disc $\bar{D^2}$ with the opposite complex structure. By gluing $D^2$ and $\bar{D^2}$ along the boundary we get a holomorphic sphere $v : S^2 \rightarrow X_P$ that we call the double of $u$. The relation $c_1(v)=\mu(u)$ then implies $\mu(u)>0$, since $J$ is spherically positive. Thus any stable map with boundary in $R_P$ has positive Maslov index. We say $(R_P, J)$ is positive.

Our goal is to prove that $R_P$ is weakly unobstructed. 
To do this we need to better understand the space of holomorphic discs with boundary on $R_P$.
The involution $\tau$ induces a map  $\tau_{\ast} : \mathcal{M}_1(R,\beta) \rightarrow \mathcal{M}_1(R,\tau_{\ast}\beta)$. It is defined as
$$ \tau_{\ast}(u)(z)= \tau (u(\bar{z}))$$
for maps $u$ defined on the disc; it can be extended to the compactification $\mathcal{M}_1(R,\beta)$ in the obvious way. Note that $\omega(\beta)=\omega(\tau_{\ast}\beta)$ and $\mu(\beta)=\mu(\tau_{\ast}\beta)$. Throughout the rest of the section we will identify elements in $\pi_2(X,R)$ that have the same energy and Maslov index. So we write $\beta=\tau_{\ast}\beta$. 

If $\tau_*$ did not have fixed points, it would imply that stable maps come in pairs and it would follow (since we are working mod 2) that $m_0=0$ in $C(R,\unov)$. This argument fails since $\tau_{\ast}$ does indeed have fixed points. Nevertheless we have a partial classification of these fixed points.
\begin{lemma}\cite[Lemma 40.10]{fooo8} \label{fixed}
Let $w$ be an element in the interior of $\mathcal{M}_1(R,\beta)$ such that $\tau_{\ast}(w)=w$. Then there exists $\beta'$ and $w' \in \mathcal{M}_1(R,\beta')$ such that 
$$\beta'+\tau_{\ast}\beta'= 2\beta' = \beta \ \textrm{and} \ w = D(w').$$

Moreover there is $w''$ and $l$ such that $w=D^l(w'')$ and $\tau_{\ast}(w'') \neq w''.$ 
(See \cite[Section 40]{fooo8} for the definition of $D$.)
\end{lemma} 

It follows from the first part of the lemma that if $\mu(\beta)=1$ then $\tau_{\ast}$ does not have fixed points on the interior of $\mathcal{M}_1(R,\beta)$. $(R,J)$ is positive, so this moduli space has no boundary (no bubbling can occur). 
Thus, by Proposition 41.13 in \cite{fooo8}, we can take a single-valued perturbation of $\mathcal{M}_1(R,\beta)$ invariant under $\tau_{\ast}$. 
The argument of the paragraph before the lemma then goes through to prove
\begin{equation}\label{maslov1}
m_{0,\mu(\beta)=1} = 0 \ \in \ C(R,\unovo).
\end{equation}

Using Lemma \ref{fixed} and a careful analysis of the boundary of the moduli spaces $\mathcal{M}_{k+1}(R,\beta)$ the authors of \cite{fooo8} obtain 
\begin{theorem}\label{unobstructed}\cite[Theorem 43.28]{fooo8} 
$R$ is unobstructed over $\unovo$. In fact there exists a bounding cochain $b_R=\sum_{i}{b_iT^{\lambda_i}}$ such that $\dim{b_i}>n$ and
$$\mathfrak{P}(R,b_R)=0$$
$$HF(R,b;\unovo) \cong H(R, \Z_2) \otimes \unovo.$$
\end{theorem}
\begin{proof}
Here we just show how to prove the statement about the dimensions of the $b_i$'s. 
This theorem is proved by induction on energy. 
Namely, they find $b(\beta_i) \in C_{n+\mu(\beta_i)-1}(R, \Z_2)$ for each $\beta_i \in \pi_2(M,R)/\sim $ with holomorphic representatives, satisfying
\begin{equation}\label{sum} 
m_{1,0}(b(\beta_i))= m_{0,\beta_i}+\sum{m_{k,\beta_{i(0)}}(b(\beta_i(1)),\ldots, b(\beta_{i(k)}))}.
\end{equation}
The sum is over all $\beta$'s satisfying $\beta_{i(0)} + \ldots + \beta_{i(k)} = \beta$.

We want to show that there are no $ b(\beta_i) \in C_{\leq n}(R, \Z_2)$. 
By definition and positivity of $(R,J)$, this could only happen when $\mu(\beta_i)=1$. In that case, (\ref{sum}) reduces to
$$m_{1,0}(b(\beta_i))= m_{0,\beta_i}.$$
But we saw before that  $m_{0,\beta_i}=0$. So we can take $b(\beta_i)=0$.
\end{proof}

\begin{rmk}
The discussion above applies to any spherically positive symplectic manifold with an anti-holomorphic involution.
\end{rmk}

We now turn our attention to the Lagrangian $(L_c, \mathcal L\rho)$.
\begin{theorem}\label{torus}
Let $X_P$ be a Fano toric manifold and $(L_c, \mathcal L\rho)$ a torus fiber with a locally constant sheaf as in (\ref{sheaf}).
Then $(L_c,\mathcal L_\rho)$ is weakly unobstructed and there exists a bounding cochain $b_c=\sum_{i\geq 0}{b_i T^{\lambda_i}}$ with $\dim{b_i}>n$ for all $i$ such that
\begin{equation}
\mathfrak{P}(L_c, \rho, b_c) = \sum^{m}_{j=1} \rho^{v_j}T^{E_j}.
\end{equation}
\end{theorem}
\begin{proof}
This proposition follows from Theorem 3.6.18 in \cite{fooo}, which says that the only obstructions to the existence of a weak bounding cochain come from $m_{0,\beta}$ with $\mu(\beta)\leq 1$. When these vanish we can choose $b$ with dimension greater than n, and then
$$m_{0,\mu(\beta)=2}=\mathfrak{P}(L_c, \rho, b_c)[L_c].$$ 

>From \ref{discs} we have that $m^{\rho}_{0,\beta}=0$ when $\mu(\beta) \leq 0$ and $\sum_{\mu(\beta)=2} m^{\rho}_{0,\beta}= \sum^{m}_{i=1} \rho^{v_j}T^{E_i}[L_c]$. To complete the proof we just need to observe that $L_c$ is orientable, thus $\mu(\beta)$ is always even. 
\end{proof}

\begin{rmk}
In \cite{toricfooo} the authors prove that any element in $H^1(L_c)$ is a weak bounding cochain in the canonical model. The dimension condition on the bounding cochain we constructed above implies that it corresponds to zero in the canonical model.
\end{rmk}
\begin{rmk}
Assume that $L$ and $L'$ are monotone.
Then using the same argument as in the previous proof we can find $b$ and $b'$ as above. Then $\mathfrak{P}(L,\rho,b)$ is precisely the obstruction $o(L,\rho)$ as defined in Section 3.
Thus, with $\Z_2$ coefficients, the condition $\mathfrak{P}(L,\rho,b)+\mathfrak{P}(L',\rho',b')=0$ is the same condition as $o(L,\rho)=o(L',\rho')$ which is the condition needed to define Floer cohomology in the monotone setting.
\end{rmk}

For the bounding cochains constructed above, Cho and Oh compute the Floer cohomology of $(L_c, \mathcal{L}_{\rho})$:
\begin{theorem}
Let $Z=(Z_1,\ldots,Z_n)$ be as in (\ref{zed}). If $Z = 0$, then 
$$HF(L_c,\mathcal{L}_{\rho};\unovof)= H^{\ast}(L_c, \Z_2)\otimes \unovof .$$
Otherwise $HF(L_c,\mathcal{L}_{\rho};\unovof)$ vanishes.
\end{theorem}

With these results we can finally show that our description of the Floer complex in Section 4 agrees with the general one.

\begin{theorem}
Consider the Lagrangians $R_P$ and $(L_c,\mathcal{L}\rho)$. For the weak bounding cochains $b_P$ and $b_c$ described in Theorems \ref{unobstructed} and \ref{torus} we have:
\begin{enumerate}
	\item The Floer cohomology $HF((R,b_R),(L_c,\mathcal{L}_{\rho},b_c);\unovf)$ is defined if and only if $o(L_c, \rho)=0$;
	\item The differential $\delta_{b_P,b_c}$ coincides with $\delta$ defined in (\ref{formula}).
\end{enumerate}
\end{theorem}
\begin{proof}
Theorem \ref{obs} implies that the Floer cohomology is defined iff $\mathfrak{P}(R,b_R)+\mathfrak{P}(L_c,\rho,b_c)=0$. But
$$\mathfrak{P}(R,b_R)+\mathfrak{P}(L_c,\rho,b_c)= \sum^{m}_{j=1} \rho^{v_j}T^{E_j}=o(L_c, \rho),$$
by Theorems \ref{unobstructed} and \ref{torus}. Thus we have proved the first part of the theorem.

For the second part, note that $(R_P,b_R)$ and $((L_c,\mathcal{L}\rho),b_c)$ satisfy the conditions of Proposition \ref{diff}. So we have $\delta_{b_P,b_c}=n_{0,0}$. So we only have to consider the moduli spaces $\mathcal{M} (p,q; B)$, with $\mu(B)=1$.
A priori an element $u \in \mathcal{M} (p,q; B)$ is a stable map: the domain of the map could be a nodal curve, not just a strip. However, since all strips, discs and spheres have positive Maslov (or Chern) index, if $\mu(B)=1$ there is only one possibility for $u$ other than a strip: a constant strip attached to a disc of Maslov index one with boundary on $R$. However this cannot happen. If such  a configuration existed, by doubling the disc we would obtain a sphere of Chern number one intersecting $L_c$. But, as is shown in \cite{toricfooo}, spheres with $c_1=1$ do not intersect the torus fibers.  

This shows that the space $\mathcal{M} (p,q; B)$ consists only of strips.
Since we already know that these strips are regular we have proved the theorem.
\end{proof}

\bibliographystyle{plain}
\bibliography{imrn-bib}

\end{document}